\documentclass[a4paper, 11pt]{amsart}
\usepackage[utf8]{inputenc}
\usepackage[T1]{fontenc}
\usepackage{lmodern}
\usepackage[french,english]{babel}

\usepackage{amsmath,amssymb,amsthm}

\usepackage{framed}

\usepackage{tikz-cd}

\usepackage{rotating}

\usepackage{enumitem}

\tikzset{%
    symbol/.style={%
        draw=none,
        every to/.append style={%
            edge node={node [sloped, allow upside down, auto=false]{$#1$}}}
    }
}

\newcommand{\Set}{\ensuremath{\mathrm{\textbf{Set}}}}

\newcommand{\Cat}{\mathbf{Cat}} 

\newcommand{\Pol}{\mathbf{Pol}} 

\newcommand{\C}{\ensuremath{\mathcal{C}}} 
\newcommand{\D}{\ensuremath{\mathcal{D}}} 
\newcommand{\W}{\ensuremath{\mathcal{W}}} 
\newcommand{\T}{\ensuremath{\mathcal{T}}} 
\newcommand{\G}{\ensuremath{\mathcal{G}}} 

\newcommand{\sD}{\ensuremath{\mathbb{D}}} 

\newcommand{\cc}{\ensuremath{\boldsymbol{\mathrm{c}}}} 
\newcommand{\ii}{\ensuremath{\boldsymbol{\mathrm{i}}}} 
\newcommand{\wt}[1]{\ensuremath{\widetilde{#1}}} 

\theoremstyle{plain}
\newtheorem{theorem}{Theorem}[section]
\newtheorem{lemma}[theorem]{Lemma}
\newtheorem{proposition}[theorem]{Proposition}

\newtheorem{corollary}[theorem]{Corollary}

\theoremstyle{definition}
\newtheorem{definition}[theorem]{Definition}
\newtheorem{remark}[theorem]{Remark}

\newtheorem{paragr}[theorem]{}
\newtheorem{example}[theorem]{Example}
\newtheorem{counterexample}[theorem]{Counter-Example}

\theoremstyle{plain}
\newtheorem*{theorem*}{Theorem} 
\newtheorem{theoremintro}{Theorem} 

\theoremstyle{definition}
\newtheorem*{definition*}{Definition} 

\title{Polygraphs and discrete Conduch\'e $\omega$-functors}
\author{L\'eonard Guetta}
\date{\today}
\email{guetta@irif.fr}
\address{IRIF, Universit\'e Paris Diderot, Case 7014, 75205 Paris CEDEX 13}
\begin{document}

\begin{abstract}
 We define a class of morphisms between strict $\omega$-categories called \emph{discrete Conduché $\omega$-functors} that generalize discrete Conduché functors between $1$-categories and we study their properties related to polygraphs. The main result we prove is that for every discrete Conduché $\omega$-functor $f : C \to D$, if $D$ is a free strict $\omega$-category on a polygraph then so is $C$.  
  \end{abstract}

\maketitle

\section*{Introduction}
In \cite[Theorem 4.4]{giraud1964methode}, Giraud introduced necessary and sufficient conditions for a functor $f : C \to D$ to be exponentiable in the category of (small) categories $\Cat$, i.e. such that the pullback functor
\[
f^{\ast} : \Cat/D \to \Cat/C
\]
induced by $f$ admits a right adjoint. A functor satisfying these conditions is usually called a \emph{Conduché functor} or \emph{Conduché fibration} (named after Conduché who rediscovered Giraud's theorem in \cite{conduche1972sujet}). In the present article, we will focus on a particular case of this notion.
\begin{definition*}
  A functor $f : C \to D$ is a \emph{discrete Conduché functor} (or \emph{discrete Conduché fibration}) if for every arrow $\gamma : x \to y$ in $C$ and every factorization
  \[
  f(\gamma) = 
  \begin{tikzcd}
  f(x) \ar[r,"\alpha"] & z \ar[r,"\beta"] & f(y),
  \end{tikzcd}
  \]
there exists a \emph{unique} factorization
    \[
    \gamma =
    \begin{tikzcd}
  x \ar[r,"\overline{\alpha}"] & \overline{z} \ar[r,"\overline{\beta}"] & y
  \end{tikzcd}
    \]
    such that $f(\overline{\alpha})=\alpha$ and $f(\overline{\beta})=\beta$.
\end{definition*}
Recall that a category $C$ is free on a graph $G$ if \[C \simeq L(G)\] where $G$ is a graph and $L$ is the left adjoint to the forgetful functor from $\Cat$ to the category of graphs.

It was remarked in \cite{street1996categorical} that discrete Conduché functors, called \emph{ulf functors} there, have some properties related to free categories on graphs. For example, the following theorem follows immediatly from the first section of op.~cit.
\begin{theorem*}
  Let $f : C \to D$ be a discrete Conduché functor. If $D$ is free on a graph then $C$ is free on a graph.
\end{theorem*}

In the setting of strict $\omega$-categories, that we shall simply call \emph{$\omega$-categories}, the notion of free category on a graph can be generalized to the notion of \emph{free $\omega$-category on a polygraph} in the terminology of \cite{burroni1993higher} (or free $\omega$-categories on a \emph{computad} in the terminology of \cite{street1976limits} or \cite{makkai2005word}).

In the present paper we shall introduce a notion of discrete Conduché functor between $\omega$-categories and prove the following generalization of the previous theorem.
\begin{theoremintro}\label{theorem:mainintro}
  Let $C$ and $D$ be $\omega$-categories and $f : C \to D$ be a discrete Conduché $\omega$-functor. If $D$ is free on a polygraph then $C$ is free on a polygraph.
  \end{theoremintro}
We will even be more precise and explicitly construct the polygraph generating $C$ from the one generating $D$.
As a by-product we will also prove the following theorem.
\begin{theoremintro}
  Let $C$ and $D$ be $\omega$-categories and $f : C \to D$ be a discrete Conduché $\omega$-functor. If $C$ is free on a polygraph and $f_n : C_n \to D_n$ is surjective for every $n\in \mathbb{N}$, then $D$ is free on a polygraph.
\end{theoremintro}
Note also that, as one would expect, discrete Conduché $\omega$-functors are exponentiable morphisms in the category of $\omega$-categories. However the proof of that fact goes beyond the scope of this paper.

The original motivation for the present paper comes from a seemingly unrelated topic. Let $\D(\mathbb{Z})$ be the localization of the category of chain complexes with respect to quasi-isomorphisms and $\Cat_{\omega}$ the category of $\omega$-categories and $\omega$-functors. In \cite{metayer2003resolutions}, Métayer defines a functor
\[
H^{\mathrm{pol}}(-) : \Cat_{\omega} \to \D(\mathbb{Z})
\]
called the \emph{polygraphic homology functor} by means of a so-called \emph{polygraphic resolution}. As it turns out, free $\omega$-categories on polygraphs are the cofibrant objects of a ``folk'' model structure on $\Cat_{\omega}$ and the polygraphic homology functor can be understood as the left derived functor of a well-known abelianization functor (see \cite{metayer2003resolutions},\cite{metayer2008cofibrant} and \cite{lafont2010folk}).

In \cite{lafont2009polygraphic}, Lafont and Métayer prove that when we restrict this functor to the category of monoids, considered as a subcategory of $\Cat$ and hence of $\Cat_{\omega}$, then it is isomorphic to the ``classical'' homology functor of monoids (which can be defined as the singular homology of the classifying space of the monoid).

While extending the previous result from monoids to $1$-categories \cite{guettapolygraphic}, I encountered the following question:

Let $ f : P \to C$ be an $\omega$-functor with $P$ a free $\omega$-category on a polygraph, $C$ a $1$-category and let $c$ be an object of $C$. Consider the $\omega$-category $P/c$ defined as the following fibred product in $\Cat_{\omega}$:
\[
\begin{tikzcd}
  P/c \ar[d] \ar[r] \ar[dr, phantom, "\lrcorner", very near start] & P \ar[d,"f"] \\
  C/c \ar[r] &C
  \end{tikzcd}
\]
where the anonymous arrow from the slice category $C/c$ to $C$ is the obvious forgetful functor. 
\begin{center}
  Question: Is $P/c$ free on a polygraph?
\end{center}
Now, it is straightforward to check that the arrow $C/c \to C$ is a discrete Conduché functor. Moreover, as we shall see, discrete Conduché functors are stable by pullback. Hence, the arrow $P/c \to P$ is a discrete Conduché $\omega$-functor. Then Theorem \ref{theorem:mainintro} provides a positive answer to the previous question.

The same strategy also yields an alternative proof of Proposition 6 of \cite{lafont2009polygraphic}. It suffices to notice that the so-called ``unfolding'' of an $\omega$-functor  \[f : P \to M,\] where $M$ is a monoid (definition 13 of op.~cit.) is just the category $P/\star$, with $\star$ the only object of $M$ when seen as a category.
\section*{Acknowledgements}
I am deeply grateful to Georges Maltsiniotis for his help and support. Any attempt to sum it up in a few words would be reductive. I am also grateful to François Métayer and Chaitanya Leena-Subramaniam for their careful reading of this article that helped improve it a lot. Finally, I would like to thank the anonymous referee for their very useful remarks and suggestions.
\section{$\omega$-categories}
This section is mainly devoted to fixing notations. Some facts are asserted and proofs are left to the reader.
\begin{paragr} An \emph{$\omega$-graph} consists of
  \begin{itemize}
  \item[-] a sequence $(C_n)_{n \in \mathbb{N}}$ of sets,
  \item[-] maps $s^k_n,t^k_n : C_n\to C_k$ for all $k< n \in \mathbb{N}$,
    \end{itemize}
    subject to the \emph{globular identities}
    \[
    s^l_n = s^l_k\circ s^k_n = s^l_k\circ t^k_n,
    \]
    \[
    t^l_n = t^l_k\circ t^k_n = t^l_k\circ s^k_n,
    \]
    whenever $l < k < n \in \mathbb{N}$. When the context is clear, we often write $s^k$ (resp.\ $t^k$) instead of $s^k_n$ (resp.\ $t^k_n$).
    
    Elements of $C_n$ are called \emph{$n$-cells}. For an $n$-cell $x$ and $k < n$, $s^k(x)$ is its \emph{$k$-source} and $t^k(x)$ its \emph{$k$-target}. When $n>0$, we use $s(x)$ (resp.\ $t(x)$) as a synonym for $s^{n-1}_n(x)$ (resp. $t^{n-1}_n(x)$). 

    Two $n$-cells $x$ and $y$ are \emph{parallel} if $n=0$ or $n>0$ and
    \[s(x)=s(y) \text{ and }t(x)=t(y).\]

    We define the set $C_n\times_{C_k}C_n$ as the following fibred product
    \[
    \begin{tikzcd}
      C_n\times_{C_k}C_n \ar[r] \ar[dr,phantom,"\lrcorner", very near start] \ar[d] &C_n \ar[d,"t^k"]\\
      C_n \ar[r,"s^k"] & C_k.
      \end{tikzcd}
    \]
    That is, elements of $C_n\times_{C_k}C_n$ are pairs $(x,y)$ of $n$-cells such that $s^k(x)~=~t^k(y)$. We say that two $n$-cells $x$ and $y$ are \emph{$k$-composable} if the pair $(x,y)$ belongs to $C_n\times_{C_k}C_n$. 
    \end{paragr}
\begin{paragr}
      Given two $\omega$-graphs $C$ and $D$, a \emph{morphism of $\omega$-graphs}
  \[
  f : C \to D
  \]
  is a sequence of maps
  \[
  (f_n : C_n \to D_n)_{n\in \mathbb{N}}
  \]
  such that, for all $k<n \in \mathbb{N}$, both diagrams
  \[
  \begin{tikzcd}
    C_n \ar[r,"f_n"] \ar[d,"s^k"] & D_n \ar[d,"s^k"] \\
    C_k \ar[r,"f_k"] & D_k
  \end{tikzcd}
    \begin{tikzcd}
    C_n \ar[r,"f_n"] \ar[d,"t^k"] & D_n \ar[d,"t^k"] \\
    C_k \ar[r,"f_k"] & D_k
    \end{tikzcd}
    \]
    are commutative.
  \end{paragr}
\begin{paragr}\label{defomegacat} An \emph{$\omega$-category} consists of an $\omega$-graph $C$ together with maps
  \[
  \ast_k^n : C_n\times_{C_k} C_n \to C_n
  \]
  \[
  1^n_k : C_k \to C_n
  \]
  for each pair $k<n$ subject to the following axioms.

  \textbf{Source and target axioms:}
  \begin{enumerate}
  \item For all $l\leq k < n \in \mathbb{N}$, for all $k$-composable $n$-cells $x$ and $y$, we have
    \[
        s^l(x\ast_k^n y)=s^l(x),
    \]
    \[
    t^l(x\ast_k^n y)=t^l(x).
    \]
  \item For all $ k <l<n \in \mathbb{N}$,  for all $k$-composable $n$-cells $x$ and $y$, we have
    \[
 s^l(x\ast_k^n y)=s^l(x)\ast_k^l s^l(y),
    \]
    \[
t^l(x\ast_k^n y)=t^l(x)\ast_k^lt^l(y).
\]
\item For all $k < n \in \mathbb{N}$, for every $k$-cell $x$, we have
  \[
  s^k(1^n_k(x))=x=t^k(1^n_k(x)).
  \]
  \end{enumerate}
  
  \textbf{Unit axioms:}
  \begin{enumerate}[resume]
  \item For all $l<k<n \in \mathbb{N}$,
    \[
    1^n_k\circ1^k_l = 1^n_l.
    \]
      \item For all $k<n \in \mathbb{N}$, for every $n$-cell $x$, we have
    \[
    x\ast_k^n1^n_k(s^k(x))=x=1^n_k(t^k(x))\ast_k^nx.
    \]
  \item For all $k<l<n \in \mathbb{N}$, for all $k$-composable $n$-cells $x$ and $y$, we have 
    \[
    1^n_l(x\ast_k^ly)=1^n_l(x)\ast^n_k1^n_l(y).
    \]
  \end{enumerate}

  \textbf{Associativity axiom:}
  \begin{enumerate}[resume]
\item For all $k<n \in \mathbb{N}$, for all $n$-cells $x,y$ and $z$ such that $x$ and $y$ are $k$-composable, and $y$ and $z$ are $k$-composable, we have
    \[
    (x\ast_k^n y)\ast_k^n z=x\ast_k^n(y \ast_k^n z).
    \]
  \end{enumerate}

  \textbf{Exchange law:}
  \begin{enumerate}[resume]
  \item For all $k < l < n \in \mathbb{N}$, for all $n$-cells $x,x',y$ and $y'$  such that
    \begin{itemize}
    \item[-] $x$ and $y$ are $l$-composable, $x'$ and $y'$ are $l$-composable,
    \item[-] $x$ and $x'$ are $k$-composable, $y$ and $y'$ are $k$-composable,
    \end{itemize}
    we have
    \[
    ((x \ast^n_k x')\ast^n_l (y \ast^n_k y'))=((x \ast^n_l y)\ast^n_k (x' \ast^n_l y')).
    \]
  \end{enumerate}
  The same letter will refer to an $\omega$-category and its underlying $\omega$-graph. We will almost always write $\ast_k$ instead of $\ast^n_k$, and, for an $n$-cell $x$, $1_x$ will sometimes be used as a synonym for $1^{n+1}_n(x)$.
  Moreover, for consistency, we set $1^n_n(x):=x$ for any $n$-cell $x$.
\end{paragr}
\begin{paragr}
  Let $C$ and $D$ be $\omega$-categories. An \emph{$\omega$-functor} is a morphism of $\omega$-graphs $f : C \to D$ that satisfies the following axioms:
  \begin{enumerate}
  \item For all $k < n \in \mathbb{N}$, for all $k$-composable $n$-cells $x$ and $y$, we have
    \[
    f_n(x \ast_k y) = f_n(x) \ast_k f_n(y).
    \]
  \item For all $k < n \in \mathbb{N}$, for every $k$-cell $x$, we have
    \[
    f_n(1^n_k(x))=1^n_k(f(x)).
    \]
    \end{enumerate}
  For an $n$-cell $x$, we will often write $f(x)$ instead of $f_n(x)$. The category of $\omega$-categories and $\omega$-functors is denoted by $\Cat_{\omega}$. 
\end{paragr}
\begin{paragr}
  Let $x$ be a $k$-cell in an $\omega$-category. We say that $x$ is \emph{degenerate} if there exists $x' \in C_{k'}$ with $k' < k$ such that
  \[
  x=1^k_{k'}(x').
  \]
  Note that $0$-cells are never degenerate.

\end{paragr}
\begin{paragr}\label{paragr:ncat}
  For $n \in \mathbb{N}$, an \emph{$n$-category} is an $\omega$-category such that every $k$-cell with $k > n$ is degenerate. An \emph{$n$-functor} is an $\omega$-functor between two $n$-categories. The category of $n$-categories and $n$-functors is denoted by $\Cat_n$. 

  There is an obvious inclusion functor
  \[
  \Cat_n \to \Cat_{\omega}.
  \]
  This functor has a left and a right adjoint. In the sequel, we shall only use the right adjoint, which will be denoted by
  \[
  \tau_{\leq n} : \Cat_{\omega} \to \Cat_n.
  \]
  For an $\omega$-category $C$, the $\omega$-category $\tau_{\leq n}(C)$ is obtained by removing all non-degenerate $k$-cells of $C$ such that $k>n$. 
\end{paragr}
\begin{remark}
  It follows from the axioms of $\omega$-categories and $\omega$-functors that, for $n$-categories and $n$-functors, everything involving $k$-cells such that $k>n$ can be recovered from the rest. For example, we will often consider that the data of an $n$-category $C$ only consists of 
  \begin{itemize}
  \item[-] $(C_k)_{0 \leq k \leq n}$,
  \item[-] $(s^{k}_l)_{0\leq k < l \leq n}$,
  \item[-] $(t^k_l)_{0\leq k < l \leq n}$,
  \item[-] $(\ast^l_k)_{0\leq k < l \leq n}$,
    \item[-] $(1^l_k)_{0\leq k < l \leq n}$.
  \end{itemize}
  \end{remark}
\begin{paragr}
  Let $n \in \mathbb{N}$. It follows from the definition of $\omega$-categories and $\omega$-functors that we have a functor
  \[
  \mathrm{Cell}_n : \Cat_{\omega} \to \Set
  \]
  that associates to each $\omega$-category $C$, the set $C_n$ of its $n$-cells and to an $\omega$-functor $f : C \to D $, the map $f_n : C_n \to D_n$.

  This functor is representable and we define the \emph{$n$-globe} $\sD_n$ to be the $\omega$-category representing this functor. ($\sD_n$ is in fact an $n$-category.) Here are some pictures in low dimension:
  \[
  \sD_0= \begin{tikzcd}\bullet\end{tikzcd},
  \]
  \[
  \sD_1 = \begin{tikzcd} \bullet \ar[r] &\bullet \end{tikzcd},
  \]
  \[
  \sD_2 = \begin{tikzcd}
    \bullet \ar[r,bend left=50,""{name = U,below}] \ar[r,bend right=50,""{name=D}]&\bullet \ar[Rightarrow, from=U,to=D]
  \end{tikzcd},
  \]
    \[
  \sD_3 = \begin{tikzcd}
    \bullet \ar[r,bend left=50,""{name = U,below,near start},""{name = V,below,near end}] \ar[r,bend right=50,""{name=D,near start},""{name = E,near end}]&\bullet \ar[Rightarrow, from=U,to=D, bend right,""{name= L,above}]\ar[Rightarrow, from=V,to=E, bend left,""{name= R,above}]
    \arrow[phantom,"\Rrightarrow",from=L,to=R]
  \end{tikzcd}.
  \]
  We will make no distinction between an $n$-cell $x$ and the $\omega$-functor 
  \[
  x : \sD_n \to C
  \]
  associated to it.

  For $k < n \in \mathbb{N}$, the arrows $s^k_n, t^k_n$ and $1^n_k$ induce natural tranformations
  \[
  \sigma^k_n,\tau^k_n : \mathrm{Cell}_n \Rightarrow \mathrm{Cell}_k
  \]
  and
  \[
  \kappa^n_k :  \mathrm{Cell}_k \Rightarrow \mathrm{Cell}_n.
  \]
  These natural transformations are in turn represented by  $\omega$-functors, that we denote with the same letters:
  \[
  \sigma^k_n, \tau^k_n : \sD_k \to \sD_n
  \]
  and
  \[
  \kappa^n_k : \sD_n \to \sD_k.
  \]
  For example, having  a commutative triangle
  \[
  \begin{tikzcd}
    \sD_n \ar[r,"x"] \ar[d,"\kappa^n_k"'] & C\\
    \sD_k \ar[ru,"y"']&
    \end{tikzcd}
  \]
  means exactly that we have an $n$-cell $x$ of $C$ and a $k$-cell $y$ of $C$ such that \[x=1^n_k(y).\]
\end{paragr}
\begin{paragr}
  Similarly, for $k < n \in \mathbb{N}$, we have a functor
  \[
  \mathrm{Comp}^n_k : \Cat_{\omega} \to \Set
  \]
  that associates to each $\omega$-category $C$ the set $C_n \times_{C_k} C_n$ and to an $\omega$-functor $f : C \to D$ the canonically induced map \[f_n \times_{f_k} f_n : C_n \times_{C_k} C_n \to D_n \times_{D_k} D_n.\]
  This functor is represented by the $\omega$-category \[\sD_n\coprod_{\sD_k}\sD_n\]
  (which is also an $n$-category) defined as the following amalgamated sum
  \[
  \begin{tikzcd}
    \sD_k \ar[r,"\sigma^k_n"] \ar[d,"\tau^k_n"] \ar[dr, phantom,"\ulcorner", very near end] & \sD_n \ar[d]\\
    \sD_n \ar[r]& \sD_n\coprod_{\sD_k}\sD_n.
    \end{tikzcd}
  \]
  The arrow $\ast^n_k$ induces a natural transformation
  \[
  \nabla^n_k : \mathrm{Comp}^n_k \to \mathrm{Cell}_n
  \]
  which in turn is represented by an $\omega$-functor
  \[
  \nabla^n_k : \sD_n \to \sD_n\coprod_{\sD_k}\sD_n.
  \]
  For example, having a commutative triangle
    \[
  \begin{tikzcd}
    \sD_n \ar[r,"x"] \ar[d,"\nabla^n_k"'] & C\\
    \sD_n\coprod_{\sD_k}\sD_n \ar[ru,"{(y, y')}"']&
    \end{tikzcd}
  \]
  means exactly that we have $n$-cells $x,y,y'$ of $C$ such that $y$ and $y'$ are $k$-composable and $x=y\ast_ky'$.
\end{paragr}
\section{Polygraphs}
\begin{definition}\label{defbasis}
  Let $C$ be an $\omega$-category and $\Sigma \subseteq C_k$ with $k \in \mathbb{N}$. We say that $\Sigma$ is a \emph{$k$-basis} if either $k = 0$ and
  \[
  \Sigma = C_0
  \]
  or $k > 0$ and the following universal property is satisfied: for every $k$-category $D$, every $(k- 1)$-functor \[ f : \tau_{\leq k-1}(C) \to \tau_{\leq k-1}(D)\] and every map \[ {\varphi : \Sigma \to D_k}\] such that for every $k$-cell $x$ of $C$, we have
  \[
  s(\varphi(x))=f_{k-1}(s(x)) \text{ and } t(\varphi(x))=f_{k-1}(t(x)),
  \]
  there exists a unique $k$-functor $f' : \tau_{\leq k}(C) \to D$ such that
  \[
  \tau_{\leq k-1}(f')=f
  \]
  and
  \[
  f'_k(x)=\varphi(x)
  \]
  for every $x \in \Sigma$.
\end{definition}
\begin{remark}
  Note that an $n$-category has a $k$-basis for any $k > n$, namely the empty set.
\end{remark}
\begin{definition}
  An $\omega$-category $C$ is \emph{free} if it has a $k$-basis for every $k \in \mathbb{N}$.
\end{definition}
\begin{definition}
  Let $C$ be an $\omega$-category and $n \in \mathbb{N}$. An $n$-cell $x$ of $C$ is \emph{indecomposable} if: 
  \begin{enumerate}
  \item $x$ is not degenerate,
    \item for any $0 \leq k <n $, if $x=x_1\ast_k x_2$ with $x_1,x_2 \in C_n$ then
    \[x_1=1^n_k(t^k(x))\]
    or
     \[x_2=1^n_k(s^k(x)).\]
    \end{enumerate}
\end{definition}
In particular, any $0$-cell is indecomposable. 
\begin{proposition}\label{propositionbasisindecomposablecells}
  Let $C$ be a free $\omega$-category. For each $k \in \mathbb{N}$, there is a unique $k$-basis of $C$, namely the set of indecomposable $k$-cells.
\end{proposition}
\begin{proof}
 See \cite[Section 4, Proposition 8.3]{makkai2005word}.
\end{proof}
\begin{paragr}
Proposition \ref{propositionbasisindecomposablecells} allows us to talk about \emph{the} $k$-basis of a free category $C$. The sequence
  \[
  (\Sigma_k \subseteq C_k)_{k \in \mathbb{N}}
  \]
  where each $\Sigma_k$ is the $k$-basis of $C$ is simply called the \emph{basis} of $C$.
  \end{paragr}
\begin{definition}
  An $\omega$-functor $f : C \to D$ between two free $\omega$-categories is \emph{rigid} if for every $k \in \mathbb{N}$, we have
  \[
  f_k(\Sigma_k^C)\subseteq \Sigma_k^D
  \]
  where $\Sigma^C_k$ (resp.\ $\Sigma_k^D$) is the $k$-basis of $C$ (resp.\ $D$).
\end{definition}
Free $\omega$-categories and rigid $\omega$-functors form a category denoted by $\Pol$.
\begin{remark}
  Objects of $\Pol$ are commonly called \emph{polygraphs} and morphisms of $\Pol$ are commonly called \emph{morphisms of polygraphs}. Although the terms ``polygraph'' and ``free $\omega$-category'' are synonyms, we prefer to use the former one when we think of them as objects of the category $\Pol$ and the latter one when we think of them as objects of the category $\Cat_{\omega}$.
\end{remark}

\section{Discrete Conduché $\omega$-functors}
\begin{paragr}
  Recall that given a category $\C$ and $M$ a class of arrows of $\C$, an arrow $f : X \to Y$ of $\C$ is said to be \emph{right orthogonal to $M$} if for every $m~:~A~\to~B~\in~M$ and every commutative square
  \[
  \begin{tikzcd}
  A \ar[d,"m"'] \ar[r] & X \ar[d,"f"]\\
  B \ar[r]& Y
  \end{tikzcd}
  \]
  there exists a \emph{unique} $l: B \to X$ (referred to as a \emph{lifting}) such that the diagram
    \[
  \begin{tikzcd}
  A \ar[d,"m"'] \ar[r] & X \ar[d,"f"]\\
  B\ar[ur,"l"] \ar[r]& Y.
  \end{tikzcd}
  \]
  is commutative.

\end{paragr}
\begin{definition}\label{definitionconduche}
  Let $f : C \to D$ be an $\omega$-functor. We say that $f$ is a \emph{discrete Conduché $\omega$-functor} if it is right orthogonal to the arrows
  \[
  \begin{tikzcd}
   \sD_n \ar[d, "\kappa^n_k"]\\ \sD_k
    \end{tikzcd}
  \]
  and
  \[
  \begin{tikzcd}
   \sD_n \ar[d, "\nabla^n_k"]\\ \sD_n\coprod_{\sD_k} \sD_n
    \end{tikzcd}
  \]
  for any $k,n \in \mathbb{N}$ such that $0 \leq k <n$.
\end{definition}
\begin{remark}
  Since the class of discrete Conduché $\omega$-functors is a right orthogonal class, it has many good properties. One of them is that discrete Conduché $\omega$-functors are stable by pullback.
  \end{remark}

\begin{paragr}Unfolding the Definition \ref{definitionconduche}, the right orthogonality to $\kappa^n_k$ means that for any $x \in C_n$, for any $y \in D_k$ such that
\[
f(x)=1^n_k(y)
\]
there exists a unique\footnote{Note that since the map $z \mapsto 1^n_k(z)$ is injective, the uniqueness comes for free.} $z \in C_k$ such that
\[
x=1^n_k(z)
\]
and
\[
f(z)=y.
\]
Similarly, the right orthogonality to $\nabla^n_k$ means that for any $x \in C_n$, if
\[
f(x)=y_1\ast_ky_2
\]
with $y_1,y_2 \in D_n$ that are $k$-composable, then there exists a \emph{unique} pair $(x_1,x_2)$ of elements of $C_n$ such that
\begin{enumerate}
\item $s^k(x_1)=t^k(x_2)$ and $x=x_1 \ast_k x_2$,
  \item $f(x_1)=y_1$ and $f(x_2)=y_2$.
\end{enumerate}
\end{paragr}

As it turns out, there are redundancies in the definition of discrete Conduché $\omega$-functor.
\begin{lemma}\label{lemmaconducheimpliesdiscrete}
Let $k<n \in \mathbb{N}$ and $f : C \to D$ an $\omega$-functor. If $f$ is right orthogonal to
    \[
  \begin{tikzcd}
   \sD_n \ar[d, "\nabla^n_k"]\\ \sD_n\coprod_{\sD_k} \sD_n
    \end{tikzcd}
  \]
  then it is right orthogonal to
  \[
  \begin{tikzcd}
   \sD_n \ar[d, "\kappa^n_k"]\\ \sD_k.
    \end{tikzcd}
  \]  
\end{lemma}
\begin{proof}
  Let $x \in C_n$ and suppose that $f(x)=1_k^n(y)$ with $y \in D_k$. Notice that
  \[f(x)=1_k^n(y)\ast_k 1_k^n(y)\]
  and
  \[x=x\ast_k1^n_k(s^k(x)) = 1^n_k(t^k(x))\ast_kx\]
  and
  \[f(1^n_k(s^k(x)))=1^n_k(s^k(f(x)))=1^n_k(y)=1^n_k(t^k(f(x)))=f(1^n_k(t^k(x))).\]
  Using the uniqueness part of the right orthogonality to $\nabla^n_k$, we deduce that $x=1^n_k(s^k(x))=1^n_k(t^k(x))$. Thus, if we set $z=s^k(x)=t^k(x)$, we have $x=1^n_k(z)$ and $f(z)=y$, which is what we needed to prove. 
\end{proof}
\begin{remark}
  In light of the previous lemma, the reader may wonder why we included the right orthogonality to $\kappa^n_k$ in the definition of discrete Conduché $\omega$-functor. The motivation for such a choice is that it \emph{should} be possible to apply this definition \emph{mutatis mutandis} for non-strict $\omega$-categories where Lemma \ref{lemmaconducheimpliesdiscrete} might not hold anymore.\end{remark}

\begin{lemma}\label{lemmaconduchecondition}
  Let $k<m<n \in \mathbb{N}$ and $f : C \to D$ be an $\omega$-functor. If $f$ is right orthogonal to $\nabla^n_k$ and $\kappa^n_m$, then it is right orthogonal to $\nabla^m_k$.
\end{lemma}
\begin{proof} Let $x \in C_m$ and suppose that \[f(x)=y_1\ast_ky_2\] with $y_1, y_2 \in D_m$ that are $k$-composable. Then $1^n_m(x) \in C_n$ and \[f(1^n_m(x))=1^n_m(y_1)\ast_k1^n_m(y_2).\]
  From the right orthogonality to $\nabla^n_k$, we know that there exist $z_1, z_2~\in~C_n$ that are $k$-composable such that $f(z_1)=1^n_m(y_1)$, $f(z_2)=1^n_m(y_2)$ and \[1^n_m(x)=z_1\ast_kz_2.\]

  From the right orthogonality to $\kappa^n_m$, we know that there exist $x_1, x_2 \in C_m$ such that $z_1=1^n_m(x_1)$, $z_2=1^n_m(x_2)$, $f(x_1)=y_1$ and $f(x_2)=y_2$. It follows that $s^k(x_1)=t^k(x_2)$ and \[1^n_m(x)=1^n_m(x_1)\ast_k1^n_m(x_2)=1^n_m(x_1\ast_kx_2),\] hence $x=x_1\ast_k x_2$. This proves the existence part of the right orthogonality to $\nabla^m_k$.
  
  Now suppose that there are two pairs $(x_1,x_2)$ and $(x_1',x_2')$ of $m$-cells of $C$ that lift the pair $(y_1,y_2)$ in the usual way. It follows that $(1^n_m(x_1),1^n_m(x_2))$ and $(1^n_m(x_1'),1^n_m(x_2'))$ lift the pair $(1^n_m(y_1),1^n_m(y_2))$ in the usual way.

  From the uniqueness part of the right orthogonality to $\nabla^n_k$, we deduce that ${1^n_m(x_1)=1^n_m(x_1')}$ and $1^n_m(x_2)=1^n_m(x_2')$, hence $x_1=x_1'$ and $x_2=x_2'$. 
\end{proof}
\begin{paragr}
  Recall that with the definition we chose (paragraph \ref{paragr:ncat}), an $n$-functor is a particular type of $\omega$-functor. Hence, it makes sense to call an $n$-functor a \emph{discrete Conduché $n$-functor} when it is a discrete Conduché $\omega$-functor.
  \end{paragr}
  \begin{proposition}
    Let $f : C \to D$ be an $n$-functor. It is a discrete Conduché $n$-functor if and only if it is right orthogonal to $\nabla^n_k$ for any $k \in \mathbb{N}$ such that $k<n$.
    \end{proposition}
  \begin{proof}
    From Lemma \ref{lemmaconducheimpliesdiscrete}, we know that we only need to show that $f$ is right orthogonal to $\nabla^m_k$ for all $k<m \in \mathbb{N}$.
    When $m\leq n$, this follows from Lemma \ref{lemmaconducheimpliesdiscrete} and Lemma \ref{lemmaconduchecondition}. As for the case $m>n$, this follows from the fact that for any $m$-cell $x$ in an $n$-category with $m>n$, there exists a unique $n$-cell $x'$ such that $x=1^m_n(x')$. Details are left to the reader.
  \end{proof}
  \begin{corollary}\label{corollary:truncationconduche}
    Let $f : C \to D$ be an $\omega$-functor and $n \in \mathbb{N}$. The $n$-functor $\tau_{\leq n }(f) : \tau_{\leq n }(C) \to \tau_{\leq n }(D)$ is a discrete Conduché $n$-functor if and only if $f$ is right orthogonal to $\nabla^n_k$ for any $k \in \mathbb{N}$ such that $k<n$. 
    \end{corollary}
\begin{lemma}\label{lemmaconducheindecomposablecells}
  Let $f : C \to D$ be a discrete Conduché $\omega$-functor and $x$ a cell of $C$. Then $x$ is an indecomposable cell if and only if $f(x)$ is an indecomposable cell.
\end{lemma}
\begin{proof}
  If $x$ is a 0-cell, there is nothing to show since every $0$-cell is indecomposable. We suppose now that $x$ is an $n$-cell with $n>0$.

  Suppose that $x$ is indecomposable. The right orthogonality to $\kappa_k^n$ for any $0 \leq k <n$ implies that $f(x)$ is non-degenerate since, if it were degenerate, $x$ would be too. Suppose that
  \[
  f(x)=y_1\ast_ky_2
  \]
  with $y_1,y_2 \in D_n$ that are $k$-composable. The right orthogonality to $\nabla^n_k$ implies that
  \[
  x=x_1 \ast_k x_2
  \]
  with $f(x_1)=y_1$ and $f(x_2)=y_2$. Since $x$ is indecomposable, $x_1$ or $x_2$ has to be of the form $1^n_k(z)$ with $z \in C_k$. Thus, $y_1$ or $y_2$ has to be of the form $1^n_k(z')$ with $z' \in D_k$. This proves that $f(x)$ is indecomposable.

  Suppose that $f(x)$ is indecomposable. Then $x$ is non-degenerate because otherwise $f(x)$ would be degenerate. Suppose that
  \[
  x=x_1 \ast_k x_2
  \]
  with $x_1,x_2 \in C_n$ that are $k$-composable. Thus,
  \[
  f(x)=f(x_1)\ast_k f(x_2).
  \]
  Since $f(x)$ is indecomposable, either $f(x_1)$ or $f(x_2)$ has to be of the form $1^n_k(z)$ with $z \in D_k$. From the right orthogonality to $\kappa^n_k$, it follows that either $x_1$ or $x_2$ has to be of the form $1^n_k(z')$ with $z' \in C_k$. This proves that $x$ is indecomposable.
\end{proof}
  From the previous lemma and Proposition \ref{propositionbasisindecomposablecells}, we deduce the following proposition. 
\begin{proposition}\label{prop:conducheimpliesrigid}
  Let $f : C \to D$ be an $\omega$-functor with $C$ and $D$ free $\omega$-categories. If $f$ is a discrete Conduché $\omega$-functor then $f$ is rigid.
  \end{proposition}

\section{Cellular extensions and technicalities on words}

\begin{definition}\label{definitioncellextension}
A \emph{cellular extension} of an $n$-category $C$ is a quadruple $(C,\Sigma,\sigma,\tau)$ where:
  \begin{itemize}
  \item[-] $C$ is an $n$-category,
  \item[-] $\Sigma$ is a set,
  \item[-] $\sigma$ and $\tau$ are maps $\Sigma \to C_n$ such that for every element $x$ of $\Sigma$, the $n$-cells $\sigma(x)$ and $\tau(x)$ are parallel.
  \end{itemize}
\end{definition}
When the natural number $n$ is understood, a \emph{cellular extension} means a cellular extension of some $n$-category.

\begin{definition}
  Let $E = (C,\Sigma,\sigma,\tau)$ and $E'=(C',\Sigma',\sigma',\tau')$ be two cellular extensions of $n$-categories. A \emph{morphism of cellular extensions} from $E$ to $E'$ is a pair $(f,\varphi)$ where:
  \begin{itemize}
  \item[-] $f$ is $n$-functor from $C$ to $C'$,
  \item[-] $\varphi$ is a map $\Sigma \to \Sigma'$,
  \item[-] the following squares are commutative
    \[
    \begin{tikzcd}
      \Sigma \ar[r,"\varphi"] \ar[d,"\sigma"] & \Sigma' \ar[d,"\sigma'"] \\
      C_n \ar[r,"f_n"] & C'_n
    \end{tikzcd}
    \qquad
       \begin{tikzcd}
      \Sigma \ar[r,"\varphi"] \ar[d,"\tau"] & \Sigma' \ar[d,"\tau'"] \\
      C_n \ar[r,"f_n"] & C'_n.
    \end{tikzcd}
    \]
    \end{itemize}
\end{definition}
Cellular extensions of $n$-categories and morphisms between them form a category $\Cat_n^+$. There is an obvious functor $U_n : \Cat_{n+1} \to \Cat_n^+$ that sends an $(n+1)$-category $C$ to the cellular extension $(\tau_{\leq n}(C),C_{n+1},s,t)$. We shall see later that this functor has a left adjoint. 
\begin{paragr}\label{paragrdefinitionwords}
Let $E=(C,\Sigma,\sigma,\tau)$ be a cellular extension of an $n$-category. We consider the alphabet that has:
\begin{itemize}
\item[-] a symbol $\cc_{\alpha}$ for each $\alpha \in \Sigma$,
\item[-] a symbol $\ii_{x}$ for each $x \in C_n$,
\item[-] a symbol $\ast_k$ for each $0 \leq k \leq n$,
\item[-] a symbol of opening parenthesis $($,
\item[-] a symbol of closing parenthesis $)$.
\end{itemize}
We write $\mathcal{W}[E]$ for the set of finite words on this alphabet. If $w$ and $w'$ are elements of $\mathcal{W}[E]$, we write $ww'$ for their concatenation.

The \emph{length} of a word $w$, denoted by $\mathcal{L}(w)$, is the number of symbols that appear in $w$.
\end{paragr}

\begin{paragr}
We now recursively define the set $\T[E] \subseteq \W[E]$ of \emph{well formed words} (or \emph{terms}) on this alphabet together with maps $s,t : \T[E]  \to C_n$ that satisfy the globular conditions: 
\begin{itemize}
\item[-] $(\cc_{\alpha}) \in \T[E]$ with $s((\cc_{\alpha}))=\sigma(\alpha)$ and $t((\cc_{\alpha}))=\tau(\alpha)$ for each $\alpha \in \Sigma$,
\item[-] $(\ii_{x}) \in \T[E]$ with $s((\ii_x))=t((\ii_x))=x$ for each $x \in C_n$,
\item[-] $ (v \ast_n w) \in \T[E]$ with $s((v \ast_n w))=s(w)$ and $t((v \ast_n w))=t(v)$ for $v,w \in \T[E]$ such that $s(v)=t(w)$,
  \item[-] $(v \ast_k w) \in \T[E]$ with \[s((v \ast_k w)) = s(v) \ast_k s(w)\] and \[t((v \ast_k w))=t(v)\ast_k t(w)\] for $v, w \in \T[E]$ and $0 \leq k < n$, such that $s^k(s(v))=t^k(t(w))$.
\end{itemize}
We define $s^k , t^k: \T[E] \to C_k$ as iterated source and target (with $s^n=s$ and $t^n=t$ for consistency). We say that two well formed words $v$ and $w$ are \emph{parallel} if
\[s(v)=s(w) \text{ and }t(v)=t(w)\]
and we say that they are \emph{$k$-composable} for a $k\leq n$ if
\[s^k(v)=t^k(w).\]
\end{paragr}

\emph{For the rest of the section, we fix some cellular extension of an $n$-category $E=(C,\Sigma,\sigma,\tau)$. All the words considered are elements of $\W[E]$.}
\begin{definition}
The \emph{size} of a well formed word $w$, denoted by $|w|$, is the number of symbols $\ast_k$ for any $0 \leq k \leq n$ that appear in the well formed word $w$.
\end{definition}

\begin{definition}\label{def:subword}A word $v$ is a \emph{subword} of a word $w$ if there exist words $a$ and $b$ such that $w$ can be written as
  \[w=avb.\]
\end{definition}
\begin{remark}
  Beware that in the previous definition, none of the words were supposed to be well formed. In particular, a subword of a well formed word is not necessarily well formed. 
\end{remark}
\begin{paragr} Since a word $w$ is a finite sequence of symbols, it makes sense to write $w(i)$ for the symbol at position $i$ of $w$, with $0 \leq i \leq \mathcal{L}(w)-1$.

  For any $0 \leq i \leq \mathcal{L}(w)-1$, define $P_{w}(i)$ to be the number of opening parenthesis in $w$ with position $ \leq i$ minus the number of closing parenthesis in $w$ with position $\leq i$. This defines a function
  \[
    P_{w} : \{0,\dots,\mathcal{L}(w)-1\} \to \mathbb{Z}.
  \]
\end{paragr}
\begin{remark}
  Such a counting function is standard in the literature about formal languages. For example see \cite[chapter 1, exercice 1.4]{hopcroft1979introduction}.
  \end{remark}
\begin{definition}
  A word $w$ is \emph{well parenthesized} if:
  \begin{enumerate}
    \item it is not empty,
  \item $P_w(i) \geq 0$ for any $0 \leq i \leq \mathcal{L}(w)-1$,
    \item $P_w(i) = 0$ if and only if $i = \mathcal{L}(w)-1$.
    \end{enumerate}
\end{definition}
\begin{paragr}
  It follows from the previous definition that the first letter of a well parenthesized word is necessarily an opening parenthesis and that the last letter is necessarily a closing parenthesis. Thus, the length of a well parenthesized word is not less than 2.

  Moreover, it is immediate that if $w_1$ and $w_2$ are well parenthesized words then, for any $0 \leq k \leq n$,
  \[
  (w_1 \ast_k w_2)
  \]
  is well parenthesized.
  \end{paragr}
\begin{lemma}\label{lemmawellformedispseudo}
  A well formed word is  well parenthesized.
\end{lemma}
\begin{proof}
  Let $w$ be a well formed word. We proceed by induction on $|w|$. If $|w|=0$, then $w$ is either of the form
  \[
  (\cc_{\alpha})
  \]
  or of the form
  \[
  (\ii_{x}).
  \]
  In either case, the assertion is trivial.
  Now suppose that $|w|>0$, we know by definition that
  \[
  w=(w_1 \ast_k w_2)
  \]
  with $w_1,w_2$ well formed words such that $|w_1|, |w_2| < |w|$. The desired properties follow easily from the induction hypothesis. Details are left to the reader.
\end{proof}
The converse of the previous lemma is obviously not true. However, Corollary \ref{corollarypartialconverse} below is a partial converse.
\begin{lemma}\label{lemmasubwordpseudo}
  Let $w$ be a well parenthesized word of the form
  \[
  w=(w_1 \ast_k w_2)
  \]
  with $w_1$ and $w_2$ well parenthesized words, and $0 \leq k \leq n$ and let $v$ be a subword of $w$. If $v$ is well parenthesized then one the following holds:
  \begin{enumerate}
  \item $v=w$,
  \item $v$ is a subword of $w_1$,
    \item $v$ is a subword of $w_2$.
    \end{enumerate}
\end{lemma}
\begin{proof}
  Let $a$ and $b$ be words such that
  \[
  avb=w=(w_1\ast_k w_2).
  \]
  Let $l_1, l_2, l, l_a, l_b, l_v$ respectively be the length of $w_1,w_2,w,a,b,v$. Notice that
  \[
  l_a+l_v+l_b=l=l_1+l_2+3.
  \]
  Notice that since $v$ is well parenthesized, the following cases are forbidden:
  \begin{enumerate}
  \item $l_1 \leq l_a \leq l_1 +1$,
  \item $l_2 \leq l_b \leq l_2 +1$,
  \item $l_a\geq l-1$,
    \item $l_b\geq l-1$.
  \end{enumerate}
  Indeed, the first case would imply that the first letter of $v$ is a closing parenthesis or the symbol $\ast_k$. Similarly, the second case would imply that the last letter of $v$ is an opening parenthesis or the symbol $\ast_k$. The third and fourth cased would imply that $l_v<2$ which is also impossible.

  That leaves us with the following cases:
  \begin{enumerate}
  \item $l_a=0$,
  \item $l_b=0$,
  \item $0<l_a<l_1$ and $0<l_b < l_2$,
  \item $0<l_a<l_1$ and $l_b > l_2+1$,
  \item $l_1 +1 < l_a$ and $0 < l_b < l_2$.
  \end{enumerate}
  If we are in the first case, then
  \[
  P_w(j)=P_v(j)
  \]
  for $0 \leq j \leq l_v-1$. That implies that $P_w(l_v-1)=0$ which means that $l=l_v$, hence $w=v$.
  
  By a similar argument that we leave to the reader, we can show that the second case implies that $w=v$.

  If we are in the fourth (resp.\ fifth) case, then it is clear that $v$ is a subword of $w_1$ (resp.\ $w_2$).

  Suppose now that we are in the third case. Intuitively, it means that the first letter of $v$ is inside $w_1$ and the last letter of $v$ is inside $w_2$. Notice first that
  \begin{equation}\label{usefulinequality}\tag{$\star$}
  l_a<l_1 <l_a+l_v-3,
  \end{equation}
  where the inequality on the right comes from the fact that $l_v \geq 2$ because $v$ is well formed.
  
  Besides, by definition of $P_w$,
  \[
  P_w(j)=P_v(j-l_a)+P_w(l_a)
  \]
  for $l_a \leq j < l_v+l_a$. In particular, we have
  \[
  1=P_{w_1}(l_1-1)+1=P_w(l_1-1)=P_v(l_1-1)+P_w(l_a).
  \]
  From \eqref{usefulinequality} and since $v$ is well parenthesized, we deduce that
  \[
  P_v(l_1-1)>0.
  \]
  Hence, $P_w(l_a)\leq 0$ which is impossible because $w$ is well formed and ${l_a < l-1}$.\qedhere
  \end{proof}
\begin{corollary}\label{corollarypartialconverse}
  Let $w$ be a well parenthesized word. If $w$ is a subword of a well formed word, then it is also well formed.
\end{corollary}
\begin{proof}
  Let $u$ be a well formed word such that $w$ is a subword of $u$. We proceed by induction on $|u|$. If $|u|=0$, then $u$ is either of the form
  \[
  (\cc_{\alpha})
  \]
  or of the form
  \[
  (\ii_x).
  \]
  In both cases, $w=u$ since the only well parenthesized subword of $u$ is $u$ itself.

  Suppose now that $|u|>0$. By definition,
  \[
  u=(u_1 \ast_k u_2)
  \]
  with $|u_1|,|u_2| < |u|$. By Lemmas \ref{lemmawellformedispseudo} and \ref{lemmasubwordpseudo}, we have that either:
  \begin{itemize}
  \item[-] $w=u$ in which case $w$ is well formed by hypothesis,
  \item[-] $w$ is a subword of $u_1$ and from the induction hypothesis we deduce that $w$ is well formed,
    \item[-] $w$ is a subword of $u_2$ which is similar to previous case.\qedhere
    \end{itemize}
\end{proof}
\begin{lemma}\label{lemmasubwords} Let $w$ be a well formed word of the form
  \[
  w = (w_1 \ast_k w_2)
  \]
  with $w_1$ and $w_2$ well formed words, and $0 \leq k \leq n$, and let $v$ be a subword of $w$. If $v$ is well formed, then we are in one of the following cases:
  \begin{enumerate}
  \item $v=w$,
  \item $v$ is a subword of $w_1$,
    \item $v$ is a subword of $w_2$.
    \end{enumerate}
\end{lemma}
\begin{proof}
  This follows immediately from Lemma \ref{lemmawellformedispseudo} and Lemma \ref{lemmasubwordpseudo}.
  \end{proof}
\begin{corollary}\label{corollarysubwordsubstitution}
  Let $u$ be a well formed word of the form
  \[
  vew 
  \]
  with $v$, $w$ and $e$ words and such that $e$ is well formed. If $e'$ is a well formed word that is parallel to $e$, then the word
  \[
  ve'w 
  \]
is also well formed.
\end{corollary}
\begin{proof}
  We proceed by induction on $|u|$.
  \begin{description}
    \item[Base case] If $|u|=0$, then necessarily $v$ and $w$ are both the empty word and the assertion is trivial.

  \item[Inductive step] If $|u| \geq 1$, then
  \[
  u=(u_1 \ast_k u_2)
  \]
  with $u_1$ and $u_2$ well formed words such that $|u_1|,|u_2| < |u|$. By hypothesis, $e$ is a subword of $u$ and from Lemma \ref{lemmasubwords}, we are in one of the following cases. 
  \begin{itemize}
  \item[-] $u=e$ in which case the assertion is trivial.
  \item[-] $e$ is a subword of $u_1$, which means that there exist words $\tilde{v}, \tilde{w}$ such that
    \[
    u_1=\tilde{v}e\tilde{w}.
    \]
    Moreover, we have \[v=(\tilde{v}\] and \[w=\tilde{w}\ast_k u_2).\] By induction hypothesis, the word
    \[
    \tilde{v}e'\tilde{w}
    \]
    is well formed and thus
    \[
    (\tilde{v}e'\tilde{w}\ast_k u_2)=vew
    \]
    is well formed.
    \item[-] $e$ is a subword of $u_2$, which is symmetric to the previous case.\qedhere
  \end{itemize}
  \end{description}
\end{proof}
\begin{lemma}\label{lemmaunicitydecompositionpseudo}
  Let $w_1,w_2,w_1',w_2'$ be well parenthesized words, and $0 \leq k \leq n$ and $0 \leq k' \leq n$ be such that
  \[
  (w_1\ast_k w_2) = (w_1' \ast_{k'} w_2').
  \]
  Then $w_1=w_1'$, $w_2=w_2'$ and $k=k'$.
\end{lemma}
\begin{proof}
  Let us define $l:=\mathrm{min}(\mathcal{L}(w_1),\mathcal{L}(w_1'))$. Notice that
  \[
  P_w(j)=P_{w_1}(j-1)+1=P_{w_1'}(j-1)+1
  \]
  for $0 < j \leq l$ hence
  \[
  P_{w_1}(l-1)=P_{w_1'}(l-1).
  \]
Since $w_1$ and $w_1'$ are well parenthesized, one of the members of the last equality (and thus both) is equal to $0$. That implies that $\mathcal{L}(w_1)=\mathcal{L}(w_1')$ and the desired properties follow immediately from that.
\end{proof}

\begin{lemma}\label{lemmaunicitydecomposition}
  Let $w_1,w_1',w_2,w_2'$ be well formed words, and $0 \leq k \leq n$ and $0 \leq k' \leq n$ be such that $(w_1\ast_k w_2)$ and $(w_1'\ast_{k'} w_2')$ are well formed.
  If
  \[
  (w_1 \ast_k w_2) = (w_1' \ast_{k'} w_2'),
  \]
  then
  \[
  w_1=w_1', w_2=w_2' \text{ and }k=k'.
  \]
  \end{lemma}
\begin{proof}
This follows from Lemma \ref{lemmawellformedispseudo} and Lemma \ref{lemmaunicitydecompositionpseudo}.
\end{proof}
\begin{corollary}\label{corollarycompatiblesubwords}
  Let $w$ be a well formed word and suppose that it can be written as
  \[w=(w_1 \ast_k w_2)\]
with $w_1$ and $w_2$ well formed words and $0\leq k \leq n$. Then $s^k(w_1)=t^k(w_2)$.
\end{corollary}
\begin{proof}
  By hypothesis, $|w| \geq 1$. From the definition of well formed words, we know that $w$ is of the form
  \[
  (w_1'\ast_{k'}w_2')
  \]
  with $w_1'$ and $w_2'$ well formed words and $0 \leq k' \leq n$ such that
  \[
  s^{k'}(w_1')=t^{k'}(w_2').
  \]
  From Lemma \ref{lemmaunicitydecomposition}, we have that $w_1'=w_1$, $w_2'=w_2$ and $k=k'$.
\end{proof}
\section{From cellular extensions to free $\omega$-categories}
\begin{definition}\label{definitionelementarymove}
  Let $E=(C,\Sigma,\sigma,\tau)$ be a cellular extension of an $n$-category and let $u,u' \in \T[E]$. An \emph{elementary move} from $u$ to $u'$ is a quadruple $\mu=(v,w,e,e')$ with $v,w \in \W[E]$ and $e,e' \in \T[E]$ such that 
  \[
  u=vew,
  \]
  \[
  u'=ve'w,
  \]
and one of the following holds:
\begin{enumerate}
  
\item
$e$ is of the form
      \[((x\ast_ky)\ast_kz)\]
  and $e'$  is of the form
      \[
      (x\ast_k(y\ast_kz))\]
      with $x,y,z \in \T[E]$ and $0\leq k \leq n$,
    \item $e$ is of the form \[((\ii_c)\ast_kx)\]
      and $e'$ is of the form
      \[x\]
      with $x \in \T[E]$, $0 \leq k \leq n$ and $c=1^n_k(t^{k}(x))$,
    \item $e$ is of the form
      \[(x\ast_k(\ii_{c}))\]
      and $e'$
      is of the form
      \[x\]
      with $x \in \T[E]$, $0 \leq k \leq n$ and $c=1^n_k(s^{k}(x))$,
  \item $e$ is of the form \[((\ii_{c})\ast_k(\ii_{d}))\] and $e'$ is of the form \[(\ii_{c\ast_kd})\]with $c,d \in C_{n}$ and $0 \leq k < n$,
  \item $e$ is of the form \[((x\ast_ky)\ast_l(z\ast_kt))\]
    and $e'$ is of the form 
    \[((x\ast_lz)\ast_k(y\ast_lt))\]
    with $x,y,z,t \in \T[E]$, $0\leq l < k \leq n$.
    \end{enumerate}
    \end{definition}

\begin{paragr}
We will use the notation
  \[
  \mu : u \rightarrow u'
  \]
  to say that $\mu$ is an elementary move from $u$ to $u'$.

  We now define an oriented graph\footnote{Here, \emph{oriented graph} is to be understood in the same way as the underlying (oriented) graph of a category.} $\G[E]$ with:
  \begin{itemize}
  \item[-] $\T[E]$ as its set of objects,
  \item[-] for all $u,u'$ in $\T[E]$, the set of elementary moves from $u$ to $u'$ as its set of arrows from $u$ to $u'$.
  \end{itemize}
  We will use the categorical notation
  \[
  \G[E](u,u')
  \]
  to denote the set of arrows from $u$ to $u'$.

    We will also sometimes write
  \[
  u \leftrightarrow u'
  \]
  to say that there exists an elementary move from $u$ to $u'$ or from $u'$ to $u$.   
\end{paragr}
\begin{definition}\label{definitionequivalence}
  Let $E=(C,\Sigma,\sigma,\tau)$ be a cellular extension of an $n$-category and $u,u' \in \T[E]$. We say that the well formed words $u$ and $u'$ are \emph{equivalent} and write
  \[
  u \sim u'
  \]
  if they are in the same connected component of $\G[E]$. More precisely, this means that there exists a finite sequence $(u_j)_{0\leq j \leq N}$ of well formed words with $u_0=u$, $u_N=u'$ and $u_j \leftrightarrow u_{j+1}$ for $0 \leq j < N$. The equivalence class of a well formed word $u$ will be denoted by $[u]$.
\end{definition}
\begin{lemma}\label{lemmaequivrelationsourcestargets}Let $u,u' \in \T[E]$. If $u \sim u'$ then $u$ and $u'$ are parallel.
\end{lemma}
\begin{proof}
  Let
  \[
  \mu =(v,w,e,e') : u \to u'
  \]
  be an elementary move from $w$ to $w'$. We are going to prove that $s(u)=s(u')$ and $t(u)=t(u')$ with an induction on $\mathcal{L}(v)+\mathcal{L}(w)$(cf.\ \ref{paragrdefinitionwords}). Notice first that, by definition of elementary moves, $|u|\geq 1$ and thus
  \[
  u=(u_1 \ast_k u_2)
  \]
  with $u_1,u_2 \in \T[E]$.
  \begin{description}
    
  \item[Base case] If $\mathcal{L}(v)+\mathcal{L}(w)=0$, it means that both $v$ and $w$ are both the empty word. It is then straightforward to check the desired property using Definition \ref{definitionelementarymove}.
  \item[Inductive step] Suppose now that $\mathcal{L}(v)+\mathcal{L}(w)\geq 0$. Since $e$ is a subword of $u$ that is well formed, from Lemma \ref{lemmasubwords} we are in one of the following cases:
    \begin{itemize}
    \item[-] $e=u$, which is exactly the base case.
    \item[-] $e$ is a subword of $u_1$, which means that there exist $\tilde{v},\tilde{w} \in \T[E]$ such that
      \[
      u_1 = \tilde{v}e\tilde{w}.
      \]
      Moreover, we have
      \[
      v=(\tilde{v}
      \]
      and
      \[
      w=\tilde{w}\ast_ku_2).
      \]
      From Corollary \ref{corollarysubwordsubstitution}, the word
      \[
      u_1':=\tilde{v}e'\tilde{w}
      \]
      is well formed. Therefore we can use the induction hypothesis on
      \[
      \tilde{\mu} :=(\tilde{v},\tilde{w},e,e') : u_1 \to u_1'.
      \]
      This shows that $s(u_1)=s(u_1')$ and $t(u_1)=t(u_1')$ and since
      \[
      u=(u_1\ast_k u_2) \text{ and } u'=(u_1' \ast_k u_2)
      \]
      it follows easily that $s(u)=s(u')$ and $t(u)=t(u')$.
      \item[-] $e$ is a subword of $u_2$, which is symmetric to the previous case.
    \end{itemize}

  \end{description}
  By definition of $\sim$, this suffices to show the desired properties.
\end{proof}
\begin{lemma}\label{lemmaequivrelationiscongruence}
  Let $v_1, v_2, v_1', v_2' \in \T[E]$ and $0 \leq k \leq n$ such that $v_1$ and $v_2$ are $k$-composable, and $v_1'$ and $v_2'$ are $k$-composable. If $v_1 \sim v_2$ and $v_1' \sim v_2'$ then
  \[
  (v_1 \ast_k v_2) \sim (v_1' \ast_k v_2').
  \]
\end{lemma}
\begin{proof}
Let 
  \[
  \mu = (v,w,e,e') : v_1 \to v_1' 
  \]
  be an elementary move. Set
  \[
  \tilde{v} := (v
  \]
  and
  \[
  \tilde{w} := w\ast_k v_2).
  \]
  Then, by definition, $(\tilde{v},\tilde{w},e,e')$ is an elementary move from $(v_1 \ast_k v_2)$ to $(v_1' \ast_k v_2)$. Similarly, if we have an elementary move from $v_2$ to $v_2'$, we obtain an elementary move from $(v_1 \ast_k v_2)$ to $(v_1 \ast_k v_2')$.

  By definition of $\sim$, this suffices to show the desired property.
\end{proof}
\begin{paragr}\label{paragr:bar}
  Let $E=(C,\Sigma,\sigma,\tau)$ be a cellular extension of an $n$-category, $D$ an $(n+1)$-category and
  \[
  (\varphi,f) : E \longrightarrow U_n(D)=(\tau_{\leq n}(D),D_{n+1},s,t)
  \]
  a morphism of cellular extensions. We recursively define a map
  \[
  \widehat{\varphi} : \T[E] \to D_{n+1} 
  \]
by
\begin{itemize}
\item[-]$\widehat{\varphi}((\cc_{\alpha}))=\varphi(\alpha)$ for $\alpha \in \Sigma$,
\item[-]$\widehat{\varphi}((\ii_x))=1_{f(x)}$ for $x \in C_n$,
\item[-]$\widehat{\varphi}((v\ast_k w))=\widehat{\varphi}(v)\ast_k\widehat{\varphi}(w)$ for $0 \leq k \leq n$, $v$ and $w$ two well formed words that are $k$-composable.
\end{itemize}
\end{paragr}
\begin{lemma}\label{lemmarho}
 The map $\widehat{\varphi}$ commutes with source and target, i.e. for a well formed word $w$, we have
\[
s(\widehat{\varphi}(w))=f(s(w)) \text{ and } t(\widehat{\varphi}(w))=f(t(w)).
\]
\end{lemma}
\begin{proof}
   The lemma is proven with an induction left to the reader.
\end{proof}
\begin{lemma}\label{lemma:barcompatible}
 Let $v$ and $w$ be two well formed words. If $v \sim w$ then \[{\widehat{\varphi}(v)=\widehat{\varphi}(w)}.\]
\end{lemma}
\begin{proof}
 It suffices to prove that for any elementary move $\mu : v \to w$, we have $\widehat{\varphi}(v)=\widehat{\varphi}(w)$, which is immediate from the axioms for $\omega$-category (see paragraph \ref{defomegacat}).
  \end{proof}
\begin{paragr}\label{paragrfreecategoryonextension}
Let $E=(C,\Sigma,\sigma,\tau)$ be a cellular extension of an $n$-category.

From Lemma \ref{lemmaequivrelationsourcestargets}, we deduce that $s,t : \T[E] \to C_n$ induce maps
\[
s, t : \T[E]/{\sim} \to C_n.
\]
Let $[v]$ and $[w]$ be two elements of $\T[E]/{\sim}$ such that $s^k([v])=t^k([w])$ for some $0 \leq k \leq n$. From Lemma \ref{lemmaequivrelationiscongruence}, we can define without ambiguity:
\[
[v]\ast_k[w] := [v\ast_k w].
\]

We leave it to the reader to show that these data add up to an $(n+1)$-category $E^*$ with $\tau_{\leq n}(E^*)=C$ and $E^*_{n+1}=\T[E]/{\sim}$. 

Note that we have a canonical map
\[
\begin{aligned}
  j_E : \Sigma &\to E^*_{n+1} \\
  \alpha &\mapsto [(\cc_{\alpha})]
  \end{aligned}
\]
and the following two triangles are commutative 
\[
\begin{tikzcd}
  \Sigma \ar[r,"j_E"] \ar[rd,"\sigma"'] & E^*_{n+1} \ar[d,"s^n"]\\
  &E_n
\end{tikzcd}
\qquad
\begin{tikzcd}
  \Sigma \ar[r,"j_E"] \ar[rd,"\tau"'] & E^*_{n+1} \ar[d,"t^n"]\\
  &E_n.
  \end{tikzcd}
\]
\begin{lemma}\label{lemmainjectivitymapj}
  The map $j_E : \Sigma \to E^*_{n+1}$ is injective.
\end{lemma}
\begin{proof}
For any $\alpha \in \Sigma$, it is straightforward to check that the number of occurences of $\cc_{\alpha}$ in a well formed word $w$ depends only on its equivalence class $[w]$. In particular, for $\alpha \neq \beta$ in $\Sigma$, $[(\cc_{\alpha})]\neq [(\cc_{\beta})]$. 
\end{proof}
As a consequence of the previous lemma, we will always consider $\Sigma$ as a subset of $E^*_{n+1}$ and $j_E$ as the canonical inclusion.
\begin{proposition}\label{universalpropertyfreecategory}
  Let $E=(C,\Sigma,\sigma,\tau)$ be a cellular extension of an $n$-category. Then $\Sigma$ is an $(n+1)$-basis of the $(n+1)$-category $E^*$.
\end{proposition}
\begin{proof}
  Let $D$ be an $(n+1)$-category, $ f : C \to \tau_{\leq n}(D)$ an $n$-functor and a map $ \varphi : \Sigma \to D_{n+1}$ compatible with source and target. In other words, we have a morphism of cellular extension $(\varphi,f):E \to U_n(D)$. From paragraph \ref{paragr:bar}, we have a map
  \[
  \widehat{\varphi} : \T[E] \to D_{n+1},
  \]
  whose restriction to $\Sigma$ is $\varphi$. From all the results from \ref{lemmaequivrelationsourcestargets} to \ref{lemma:barcompatible}, $\widehat{\varphi}$ induces a map
  \[
  \T[E]/{\sim} \to D_{n+1},
  \]
  which is compatible with source, target, units, and composition.
  This proves the existence of an $(n+1)$-functor
  \[
  f' : E^* \to D
  \]
  such that $\tau_{\leq n}(f')=f$ and the restriction of $f'_{n+1}$ to $\Sigma$ is $\varphi$.
  
  Let $f'' : E^{*} \to D$ be another $(n+1)$-functor with the same properties and let $x$ be an $(n+1)$-cell of $E^*$. By definition, there exists a well formed word $w$ such that
  \[
  x=[w].
  \]
  By a quick induction on $|w|$ that we leave to the reader, we easily prove that ${f''_{n+1}=f'_{n+1}}$. Since by definition $\tau_{\leq n}(f'')=f$, we have $f'=f''$. 
  \end{proof}
We leave it to the reader to extend the correspondance \[E=(C,\Sigma,\sigma,\tau)~\mapsto~E^*\] to a functor $\Cat_n^+ \to \Cat_{n+1}$. 
\begin{corollary}
  Let $n \in \mathbb{N}$. The functor
  \[
  \begin{aligned}
    \Cat_{n}^+ &\to \Cat_{n+1} \\
    E& \mapsto E^*
  \end{aligned}
  \]
  is left adjoint to the functor
  \[
  U_n : \Cat_{n+1} \to \Cat_{n}^+.
  \]
\end{corollary}
\begin{proof}
  This is a reformulation of the universal property from Definition \ref{defbasis}.\qedhere
\end{proof}
\begin{remark} Note that this left adjoint has already been explicitly constructed in the literature, for example in \cite{makkai2005word} or in \cite{metayer2008cofibrant}. Our construction is greatly inspired from the latter reference but it differs in one subtle point. Métayer defines an elementary relation on parallel well formed words and then takes the congruence generated by it, whereas we directly defined an explicit equivalence relation (Definition \ref{definitionequivalence}) and then showed that that two equivalent well formed words are necessarily parallel (Lemma \ref{lemmaequivrelationsourcestargets}) and that it is in fact a congruence (Lemma \ref{lemmaequivrelationiscongruence}).
\end{remark}
\begin{paragr}\label{paragrdefinitionevaluationfunction}
  Let $C$ be an $\omega$-category and $\Sigma$ a subset of $C_{n+1}$. We define the cellular extension
\[
E_{\Sigma}=(\tau_{\leq n}(C),\Sigma,s,t)
\]
where $s$ and $t$  mean the restrictions of $s,t : C_{n+1} \to C_n$ to $\Sigma \subseteq C_{n+1}$.

In order to simplify the notations, we will allow ourselves to  write $\T[\Sigma]$ instead of $\T[E_{\Sigma}]$ when there is no ambiguity on the rest of the data.

The canonical inclusion $\iota : \Sigma \hookrightarrow C_{n+1}$ induces a canonical morphism
\[
(\iota,\mathrm{id}_{\tau_{\leq n}(C)}) : E_{\Sigma} \to U_n(C)=(\tau_{\leq n}(C),C_{n+1},s^n,t^n).
\]
From paragraph \ref{paragr:bar}, we obtain a map, that we denote $\rho_{\Sigma}$ instead of $\widehat{\iota}$,
\[
\rho_{\Sigma} : \T[\Sigma] \to C_{n+1}
\]
such that
\begin{itemize}
\item[-] $\rho_{\Sigma}((\cc_{\alpha}))=\alpha$ for $\alpha \in \Sigma$,
\item[-] $\rho_{\Sigma}((\ii_x))=1_{x}$ for $x \in C_n$,
  \item[-] $\rho_{\Sigma}((v\ast_kw))=\rho_{\Sigma}(v)\ast_k\rho_{\Sigma}(w)$ for $0\leq k \leq n$, $v$ and $w$ two well formed words that are $k$-composable.
\end{itemize}
\end{paragr}

\begin{paragr}\label{paragr:fiberrho}
Let $a$ be an element of $C_{n+1}$, we define $\T[\Sigma]_a$ to be
\[
\T[\Sigma]_a=\{w\in \T[\Sigma]\quad | \quad \rho_{\Sigma}(w)=a\}.
\]
Lemma \ref{lemma:barcompatible} implies that if $v \in \T[\Sigma]_a$ and $v \sim w$, then $w \in \T[\Sigma]_a$.

We define $\G[\Sigma]_a$ to be the full subgraph of $\G[\Sigma]$ whose set of objects is $\T[\Sigma]_a$.
\end{paragr}
\begin{proposition}\label{propositionbasis0connected}
  Let $C$ be an $\omega$-category and $\Sigma \subseteq C_{n+1}$. Then $\Sigma$ is an $(n+1)$-basis of $C$ if and only if for every $a \in C_{n+1}$, the graph $\G[\Sigma]_a$ is 0-connected (i.e. non-empty and connected).
  
  More precisely this means that for every $a \in C_{n+1}$: 
  \begin{itemize}
  \item[-] there exists $w \in \T[\Sigma]$ such that $\rho_{\Sigma}(w)=a$,
  \item[-] for every $v,w \in \T[\Sigma]$, if $\rho_{\Sigma}(v)=a=\rho_{\Sigma}(w)$ then $v \sim w$.
    \end{itemize}
\end{proposition}
\begin{proof}
  From Proposition \ref{universalpropertyfreecategory}, we know that $E_\Sigma^*$ has $\Sigma$ as an $(n+1)$-basis. Hence, the canonical morphism
  \[
  E_{\Sigma} \to U_n(C)
  \]
  from paragraph \ref{paragrdefinitionevaluationfunction} induces a map
    \[
  \T[\Sigma]/{\sim}\to C_{n+1},
  \]
which is nothing but the map obtained from $\rho_{\Sigma}$ by applying Lemma \ref{lemma:barcompatible}. This maps sends $\Sigma$ (as a subset of $\T[\Sigma]/{\sim}$) to $\Sigma$ (as a subset of $C_{n+1}$). Since $\Sigma$ is an $(n+1)$-basis of $E_{\Sigma}^*$, we easily deduce that $\Sigma$ is an $(n+1)$-basis of $C$ if and only if the previous map is an isomorphism, which is exactly what we wanted to prove. Details are left to the reader.
\end{proof}
\end{paragr}

\section{Discrete Conduché $\omega$-functors and polygraphs}
\begin{paragr}
Let $f : C \to D$ be an $\omega$-functor, $n>0$, $\Sigma^C \subseteq C_n$ and $\Sigma^D \subseteq D_n$ such that $f_n(\Sigma^C)\subseteq \Sigma^D$. We recursively define a map 
\[
\widetilde{f} : \W[\Sigma^C] \to \W[\Sigma^D]
\]
with
\begin{itemize}
\item[-] $\wt{f}(\cc_{\alpha})=\cc_{f(\alpha)}$ for $\alpha \in \Sigma^C$,
\item[-] $\wt{f}(\ii_x)=\ii_{f(x)}$ for $x \in C_n$,
\item[-] $\wt{f}(\ast_k)=\ast_k$ for $0 \leq k < n$,
\item[-] $\wt{f}(\,(\,)=($,
  \item[-] $\wt{f}(\,)\,)=)$.
\end{itemize}
Notice that for any word $w \in \W[\Sigma^C]$, $|\wt{f}(w)|=|w|$ and $\mathcal{L}(\wt{f}(w))=\mathcal{L}(w)$.
\end{paragr}
\begin{lemma}\label{lemmamapinducedonwords}
 Let $f : C \to D$ be an $\omega$-functor, $\Sigma^C \subseteq C_{n}$ and $\Sigma^D \subseteq D_{n}$ such that $f_{n}(\Sigma^C)\subseteq \Sigma^D$. For every $u \in \W[\Sigma^C]$:
  \begin{enumerate}
  \item if $u$ is well formed then $\wt{f}(u)$ is well formed,
  \item if $\wt{f}(u)$ is well formed and if $u$ is a subword (\ref{def:subword}) of a well formed word then it is also well formed.
    \end{enumerate}
\end{lemma}
\begin{proof}
  The first part of the previous lemma is proved with a short induction left to the reader. For the second part, first notice that the map
  \[
  \wt{f} : \W[\Sigma^C] \to \W[\Sigma^D]
  \]
  satisfies the following property: 
  \begin{center}
    For any $w \in \W[\Sigma^C]$, $w$ is well parenthesized if and only if $\wt{f}(w)$ is well parenthesized.
    \end{center}
It suffices then to apply Lemma \ref{lemmawellformedispseudo} and then Corollary \ref{corollarypartialconverse}.
\end{proof}
\begin{paragr}
The first part of Lemma \ref{lemmamapinducedonwords} shows that $\wt{f}$ induces a map 
\[
\wt{f} : \T[\Sigma^C] \to \T[\Sigma^D].
\]
Moreover, we have a commutative square
\[
\begin{tikzcd}
  \T[\Sigma^C] \ar[r,"\rho_{C}"] \ar[d,"\wt{f}"] &C_{n} \ar[d,"f_{n}"] \\
  \T[\Sigma^D] \ar[r,"\rho_{D}"] & D_{n}
  \end{tikzcd}
\]
where $\rho_C$ and $\rho_D$ respectively stand for $\rho_{\Sigma^C}$ and $\rho_{\Sigma^D}$.

Thus, for every $a \in C_{n}$ we can define a map:
\[
\begin{aligned}
  \wt{f}_a : \T[\Sigma^C]_a &\to \T[\Sigma^D]_{f(a)} \\
  w &\mapsto \wt{f}(w).
\end{aligned}
\]
\end{paragr}
Recall from Corollary \ref{corollary:truncationconduche} that for an $\omega$-functor $f : C \to D$ and $n>0$, $\tau_{\leq n}(f)$ is a discrete Conduché $n$-functor if and only if $f$ is right orthogonal to $\nabla^n_k$ for any $k \in \mathbb{N}$ such that $k<n$.
\begin{proposition}\label{propositionequivalentconditionsconduche}
  
  Let $f : C \to D$ be an $\omega$-functor and $n>0$. Then the following conditions are equivalent:
  \begin{enumerate}
    \item $\tau_{\leq n}(f): \tau_{\leq n}(C) \to \tau_{\leq n}(D)$ is a discrete Conduché $n$-functor,
    \item for every $\Sigma^D \subseteq D_{n}$ and $\Sigma^C:=f^{-1}(\Sigma^D)$ and for every $a \in C_{n}$ the map
  \[
    \wt{f}_a : \T[\Sigma^C]_a \to \T[\Sigma^D]_{f(a)}
    \]
    defined above is bijective.

    \end{enumerate}
\end{proposition}
\begin{proof}
  We begin with $1 \Rightarrow 2$.
  \begin{description}
  \item[Surjectivity] We are going to prove the following assertion:
    \[
    \begin{aligned}
      \forall l \in \mathbb{N}, \forall a \in C_{n}, &\forall w \in \T[\Sigma^D]_{f(a)} \text{ such that }|w|\leq l \\
      &\exists v \in \T[\Sigma^C]_a \text{ such that }\wt{f}_a(v)=w.
      \end{aligned}
    \]
    We proceed by induction on $l$.

    Suppose first that $l=0$. We are necessarily in one of the two cases:
    \begin{enumerate}
    \item $w=(\cc_{\beta})$ with $\beta \in \Sigma^D$. By hypothesis,
      $\rho_D(w)=f(a)$
      and by definition of $\rho_{D}$,
      $\rho_D(w)=\beta$
      thus $f(a)=\beta$.
      By definition of $\Sigma^C$, $a\in \Sigma^C$ and we can choose
      $v =(\cc_a) $.
    \item $w=(\ii_{y})$ with $y \in D_{n-1}$.
      By hypothesis, $\rho_D(w)=f(a)$ and by definition of $\rho_D$, $\rho_D(w)=1_y$ thus $f(a)=1_y$. Since $\tau_{\leq n}(f)$ is a discrete Conduché $n$-functor, $f$ is right orthogonal to $\kappa^{n}_{n-1}$. Hence, there exists $x \in C_{n-1}$ such that $a=1_x$ and $f(x)=y$. We can then choose $v=(\ii_x) \in \T[\Sigma^C]_a$.
      \end{enumerate}
    Now suppose that the assertion is true for a fixed $l \in \mathbb{N}$ and let $w \in \T[\Sigma^D]_{f(a)}$ be such that $|w|=l+1$.

    By definition of well formed words, we have
    \[
    w=(w_1\ast_k w_2)
    \]
    with $0 \leq k < n$ and  $w_1, w_2 \in \T[\Sigma^D]$ such that $|w_1|\leq l$ and $|w_2| \leq l$.

    By hypothesis, $\rho_D(w)=f(a)$ and by definition of $\rho_D$, \[\rho_D(w)=\rho_D(w_1)\ast_k \rho_D(w_2)\]
    thus,
    \[\rho_D(w_1)\ast_k\rho_D(w_2)=f(a).\]
    Since by hypothesis $f$ is right orthogonal to $\nabla^n_k$, we know that there exist $a_1 \in C_{n}$ and $a_2 \in C_{n}$ that are $k$-composable and such that $a=a_1\ast_ka_2$, $f(a_1)=\rho_D(w_1)$ and $f(a_2)=\rho_D(w_2)$.

    Since $|w_1|\leq l$ and $|w_2| \leq l$, we can apply the induction hypothesis. Hence, there exist $v_1 \in \T[\Sigma^C]_{a_1}$ and $v_2 \in \T[\Sigma^C]_{a_2}$ such that $\wt{f}_{a_1}(v_1)=\wt{f}(v_1)=w_1$ and $\wt{f}_{a_2}(v_2)=\wt{f}(v_2)=w_2$. Since $\rho_C$ commutes with source and target by Lemma \ref{lemmarho}, $v_1$ and $v_2$ are $k$-composable and the word $(v_1\ast_kv_2)$ is a well formed. By definition of $\rho_C$, we have
    \[\rho_C((v_1 \ast_k v_2))=\rho_C(v_1)\ast_k\rho_C(v_2)=a_1\ast_k a_2=a.\]
    Thus, $(v_1 \ast_k v_2) \in \T[\Sigma^C]_a$ and
    \[\wt{f}_a((v_1\ast_kv_2)=\wt{f}((v_1 \ast_k v_2))=(\wt{f}(v_1)\ast_k \wt{f}(v_2))=(w_1\ast_k w_2)=w.\]
  \item[Injectivity] We are going to prove the following assertion:
\[
\forall l \in \mathbb{N},\forall v \in \T[\Sigma^C]_a, w \in \T[\Sigma^C]_a \text { such that } |v|=|w|\leq l\]
\[
    \wt{f}_a(v)=\wt{f}_a(w) \Rightarrow v=w
\]
    We proceed by induction on $l$.

    Suppose first that $l=0$. We are necessarily in one of the four cases:
    \begin{enumerate}
    \item $v=(\cc_{\alpha})$ and $w=(\cc_{\beta})$ with $\alpha$ and $\beta$ in $\Sigma^C$. By definition of $\rho_C$, $\alpha=\rho_C(v)=a=\rho_C(w)=\beta$. Hence, $v=w$.
    \item $v=(\ii_x)$ and $w=(\ii_y)$ with $x$ and $y$ in $C_{n-1}$. By hypothesis $\rho_C(v)=a=\rho_C(w)$ and by definition of $\rho_C$, $1_x=\rho_C(v)=a=\rho_C(w)=1_y$, thus $x=y$ and $v=w$.
    \item $v=(\cc_{\alpha})$ and $w=(\ii_x)$ with $\alpha \in \Sigma^C$ and $x \in C_{n-1}$. By hypothesis, $(\cc_{f(\alpha)})=\wt{f}(v)=\wt{f}(w)=(\ii_{f(x)})$ which is impossible.
      \item $v=(\ii_x)$ and $w=(\cc_{\alpha})$ with $\alpha \in \Sigma^C$ and $x \in C_{n-1}$, which is symmetric to the previous case.
    \end{enumerate}
    Now suppose that the assertion is true for a fixed $l\in \mathbb{N}$ and let $v,w \in \T[\Sigma^C]$ such that $|v|=|w|=l+1$ and $\wt{f}(v)=\wt{f}(w)$.
    By definition of well formed words, we have
    \[
    v=(v_1 \ast_k v_2)
    \]
    and
    \[
    w=(w_1 \ast_{k'} w_2)
    \]
    with $|v_1|,|v_2|,|w_1|,|w_2|\leq l$.
    
    By hypothesis, we have \[
    (\wt{f}(v_1)\ast_k\wt{f}(v_2))=\wt{f}(v)=\wt{f}(w)=(\wt{f}(w_1)\ast_{k'}\wt{f}(w_2)).
    \]
    From Lemma \ref{lemmaunicitydecomposition}, we deduce that $\ast_k=\ast_{k'}$ and $\wt{f}(v_j)=\wt{f}(w_j)$ for $j \in \{1,2\}$.

    In order to apply the induction hypothesis, we need to show that $\rho_C(v_j)=\rho_C(w_j)$ for $j \in \{1,2\}$.

    By hypothesis,
    \[\rho_C(v_1)\ast_k\rho_C(v_2)=\rho_C(v)=a=\rho_C(w)=\rho_C(w_1)\ast_k\rho_C(w_2).\]
    Hence,
    \[f(\rho_C(v_1))\ast_kf(\rho_C(v_2))=f(a)=f(\rho_C(w_1))\ast_kf(\rho_C(w_2)).\]
    Besides, $f(\rho_C(v_j))=\rho_D(\wt{f}(v_j))=\rho_D(\wt{f}(w_j))=f(\rho_C(w_j))$. We deduce from the fact that $f$ is right orthogonal to $\nabla^n_k$ that
    \[\rho_C(v_j)=\rho_C(w_j)\]
    for $j \in \{1,2\}$.

    From the induction hypothesis we have $v_j=w_j$ for $j \in \{1,2\}$, hence $v=w$.
  \end{description}
  Now we prove $2 \Rightarrow 1$. 

  Let $a \in C_{n}$ and suppose that $f(a)=b_1\ast_k b_2$. We set $\Sigma^D=\{b_1,b_2\}$. By definition, $((\cc_{b_1})\ast_k (\cc_{b_2})) \in \T[\Sigma^D]_{f(a)}$ and by hypothesis there exists a unique $v \in \T[\Sigma^C]_{a}$ such that $\wt{f}_a(v)=((\cc_{b_1})\ast_k(\cc_{b_2}))$. Since $|\wt{f}_a(v)|=|v|=1$, we have
  \[
  v=(v_1 \ast_{k'} v_2)
  \]
  with $|v_1|=|v_2|=0$, $s^{k'}(v_1)=t^{k'}(v_2)$ and $0 \leq k' < n$. Thus,
  \[
  (\wt{f}(v_1)\ast_{k'}\wt{f}(v_2))=\wt{f}(v)=((\cc_{b_1})\ast_k(\cc_{b_2})).
  \]
  Using Lemma \ref{lemmaunicitydecomposition}, we deduce that $k=k'$ and $\wt{f}(v_j)=(\cc_{b_j})$ for $j\in \{1,2\}$.
  
  We set $a_1=\rho_C(v_1)$, $a_2=\rho_C(v_2)$ and we have $s^k(a_1)=t^k(a_2)$, \[a=\rho_C(v)=\rho_C(v_1)\ast_k \rho_C(v_2)=a_1\ast_k a_2\] and \[f(a_j)=f(\rho_C(v_j))=\rho_D(\wt{f}(v_j))=\rho_D(\cc_{b_j})=b_j\] for $j \in \{1,2\}$, which proves the existence part of the right orthogonality to $\nabla^n_k$.

  Now suppose that we have $a_1,a_1',a_2,a_2' \in C_{n}$ with $s^k(a_1)=t^k(a_2)$, $s^k(a_1')=t^k(a_2')$, $a_1\ast_k a_2 = a_1'\ast_k a_2'=a$, $f(a_1)=f(a_1')=b_1$ and $f(a_2)=f(a_2')=b_2$.

  By definition of $\Sigma^C=f^{-1}(\Sigma^D)$, we have $a_1,a_1',a_2,a_2' \in \Sigma^C$. We set $w=((\cc_{a_1})\ast_k(\cc_{a_2}))$ and $w'=((\cc_{a_1'})\ast_k(\cc_{a_2'}))$. We have $\rho_C(w)=\rho_C(w')=a$ and $\wt{f}(w)=((\cc_{b_1})\ast_k(\cc_{b_2}))=\wt{f}(w')$. The injectivity of $\wt{f}_a$ implies that $w=w'$, hence $a_1=a_1'$ and $a_2=a_2'$ which proves the uniqueness part of the right orthogonality to $\nabla^n_k$. 
\end{proof}
\begin{paragr}\label{paragr:imageelemmove}
  Let $f : C \to D$ be an $\omega$-functor, $n > 0$, $\Sigma^C \subseteq C_{n}$ and $\Sigma^D \subseteq D_{n}$ such that $f(\Sigma^C)\subseteq \Sigma^D$. It follows from the definition of $\wt{f} : \T[\Sigma^C] \to \T[\Sigma^D]$ and the definition of elementary move (\ref{definitionelementarymove}) that for an elementary move
  \[
  \mu = (v,w,e,e') : u \to u'
  \]
  with $u,u' \in \T[\Sigma^C]$, the quadruple
  \[
  (\wt{f}(v),\wt{f}(w),\wt{f}(e),\wt{f}(e'))
  \]
  is an elementary move from $\wt{f}(u)$ to $\wt{f}(u')$. Thus, we have defined a map
  \[
  \G[\Sigma^C](u,u') \to \G[\Sigma^D](\wt{f}(u),\wt{f}(u')).
  \]
  Together with the map $\wt{f} : \T[\Sigma^C] \to \T[\Sigma^D]$, this defines a morphism of graphs
  \[
  \wt{f} : \G[\Sigma^C] \to \G[\Sigma^D]
  \]
  and, by restriction, a morphism of graphs
    \[
  \wt{f}_a : \G[\Sigma^C]_a \to \G[\Sigma^D]_{f(a)}
  \]
  for any $a \in C_{n}$.
\end{paragr}
\begin{lemma}\label{lemmafaithful} With the notations of the above paragraph, the map
  \[
  \G[\Sigma^C](u,u') \to \G[\Sigma^D](\wt{f}(u),\wt{f}(u'))
  \]
  is injective.
\end{lemma}
  \begin{proof}
    Let $(v_1,w_1,e_1,e_1')$ and $(v_2,w_2,e_2,e_2')$ be two elementary moves from $u$ to $u'$ such that
  \[
   (\wt{f}(v_1),\wt{f}(w_1),\wt{f}(e_1),\wt{f}(e_1'))= (\wt{f}(v_2),\wt{f}(w_2),\wt{f}(e_2),\wt{f}(e_2')).
  \]
 In particular, we have
  \[
  \mathcal{L}(v_1)=\mathcal{L}(v_2) \text{ , }  \mathcal{L}(w_1)=\mathcal{L}(w_2) \text{ , }  \mathcal{L}(e_1)=\mathcal{L}(e_2) \text{ , }  \mathcal{L}(e_1')=\mathcal{L}(e_2'). 
  \]
  Since
  \[
  v_1e_1w_1=u=v_2e_2w_2 \text{ and } v_1e_1'w_1=u'=v_2e_2'w_2,
  \]
  we have
  \[
    v_1=v_2 \text{ , }  w_1=w_2 \text{ , }  e_1=e_2 \text{ , }  e_1'=e_2'. \qedhere
  \]
  \end{proof}
  \begin{lemma}\label{technicallemma}
    With the notations of paragraph \ref{paragr:imageelemmove}, suppose that $\tau_{\leq n-1}(f)$ is a discrete Conduché $(n-1)$-functor. Let
  \[
  \mu : v \to v'
  \]
  be an elementary move in $\T[\Sigma^D]$. If there exists $u \in \T[\Sigma^C]$ such that
  \[
  \wt{f}(u) = v
  \]
  then there exists $u' \in \T[\Sigma^C]$ and an elementary move
  \[
  \lambda : u \to u'
  \]
  such that
  \[
  \wt{f}(u')=v' \text{ and } \wt{f}(\lambda)=\mu.
  \]
\end{lemma}
\begin{proof}
  The proof is long and tedious as we have to check all the different cases of elementary moves. For the sake of clarity, we first outline a sketch of the proof that is common to all the cases of elementary moves and then we proceed to fill in the blanks successively for each case.

 Let
  \[
  \mu =(v_1,v_2,e,e') : v \to v'
  \]
  be an elementary move. Since, by definition,
  \[\wt{f}(u)= v = v_1ev_2\]
  $u$ is necessarily of the form
  \[
  u=u_1\overline{e}u_2
  \]
  with $\overline{e},u_1,u_2 \in \W[E]$ such that
  \[
  \wt{f}(\overline{e})=e
  \]
    and
    \[
     \wt{f}(u_j)=v_j
     \]
     for $j \in \{1,2\}$.
     From the second part of Lemma \ref{lemmamapinducedonwords}, we deduce that $\overline{e}$ is well formed. In each different case, we will prove the existence of a well formed word $\overline{e'}$ parallel to $\overline{e}$ and such that
     \[
     \wt{f}(\overline{e'})=e'.
     \]
     From Corollary \ref{corollarysubwordsubstitution}, we deduce that the word
     \[
     \overline{u'}:=u_1\overline{e'}u_2
     \]
     is well formed. By definition, we have
     \[
     \wt{f}(u')=v'.
     \]
     Moreover, in each case, it will be immediate that the pair $(\overline{e},\overline{e'})$ is such that the quadruple
     \[
     \lambda := (u_1,u_2,\overline{e},\overline{e'})
     \]
    is an elementary move and that
     \[
     \wt{f}(\lambda)=\mu.
     \]

  All that is left now is to prove the existence of $\overline{e'}$ with the desired properties.

    \begin{description}
    \item[First case] $e$ is of the form
         \[((x\ast_k y )\ast_k z)\]  
      and $e'$ is of the form
       \[(x\ast_k (y \ast_k z))\]
      with $x,y,z \in \T[\Sigma^D]$.
      The word $\overline{e}$ is then necessarily of the form
      \[
      ((\overline{x}\ast_k \overline{y}) \ast_k \overline{z}).
      \]
      Since $\wt{f}(\overline{e})=e$, we deduce from Lemma \ref{lemmamapinducedonwords} that $\overline{x}$, $\overline{y}$, $\overline{z}$ and $(\overline{x} \ast_k \overline{y})$ are well formed. From Corollary \ref{corollarycompatiblesubwords}, we deduce that
           \[
     s^k(\overline{x})=t^k(\overline{y})
     \]
     and
          \[
     s^k(\overline{y})=t^k(\overline{z}).
     \]
     Thus, the word
     \[
     \overline{e'}:=(\overline{x}\ast_k(\overline{y}\ast_k\overline{z}))
     \]
    is well formed and it satisfies the desired properties.
  \item[Second case]
    $e$ is of the form
    \[
    (x\ast_k(\ii_{1^{n-1}_k(z)}))
    \]
    and $e'$ is of the form
    \[
    x
    \]
    with $x \in \T[\Sigma^C]$, $0\leq k < n$ and $z \in D_k$.\footnote{Notice that since $e$ is well formed, we deduce from Corollary \ref{corollarycompatiblesubwords} that $z=s^k(x)$.}
    
    Necessarily $\overline{e}$ is of the form
    \[
     (\overline{x}\ast_k(\ii_{y}))
    \]
    with $\overline{x} \in \T[E]$ (from  Lemma \ref{lemmamapinducedonwords} again) and $y \in C_{n-1}$ such that
    \[
    \wt{f}(\overline{x})=x
    \]
    and
    \[
    \wt{f}(y)=1^{n-1}_k(z).
    \]
    Then we set
    \[
    \overline{e'}:=x.
    \]
    The only thing left to show is that $y=1^{n-1}_k(s^k(\overline{x}))$. If $k=n-1$, this follows from Corollary \ref{corollarycompatiblesubwords} and the fact that $\overline{e}$ is well formed.
If $k <n-1$, we need first to use the fact that $f$ is right orthogonal to $\kappa^{n-1}_k$ to deduce that 
    \[
    y=1^{n-1}_k(\overline{z})
    \]
    for some $\overline{z} \in C_k$ such that $f(\overline{z})=z$ and then use Corollary \ref{corollarycompatiblesubwords} and the fact that $\overline{e}$ is well formed.
    \item[Third case] Similar to the second one with unit on the left.
  \item[Fourth case] $e$ is of the form
    \[
    ((\ii_x) \ast_k (\ii_y))
    \]
    and $e'$ is of the form
    \[
    (\ii_{x\ast_k y})
    \]
    with $x,y \in D_{n-1}$ such that $s^k(x)=t^k(y)$.
    Necessarily, $\overline{e}$ is of the form
    \[
    ((\ii_{\overline{x}})\ast_k(\ii_{\overline{y}}))
      \]
      with $\overline{x},\overline{y} \in C_n$ such that
      \[
      f(\overline{x})=x \text{ and } f(\overline{y})=y.
      \]
      Using Corollary \ref{corollarycompatiblesubwords} and the fact that $e$ is well formed, we deduce that $s^k(\overline{x})=t^k(\overline{y})$. Thus, the word
      \[
      \overline{e'}:=(\ii_{\overline{x} \ast_k \overline{y}})
      \]
      is well formed. It satisfies all the desired properties.
    \item[Fifth case] $e$ is of the form
      \[
      ((x\ast_ky)\ast_l(z \ast_k t))
      \]
      and $e'$ is of the form
            \[
      ((x\ast_lz)\ast_k(y \ast_l t))
      \]
      with $x,y,z,t \in \T[\Sigma^D]$ and $0 \leq l  < k < n$ such that all the compatibilities of sources and targets needed are satisfied.  

      Necessarily, $\overline{e}$ is of the form
      \[
      ((\overline{x}\ast_k \overline{y})\ast_l (\overline{z} \ast_k \overline{t}))
      \]
      with $\overline{x},\overline{y},\overline{z},\overline{t} \in \W[\Sigma^C]$ such that
      \[
          \wt{f}(\overline{x})=x,
    \]
    \[
    \wt{f}(\overline{y})=y,
    \]
    \[
    \wt{f}(\overline{z})=z,
    \]
    \[
    \wt{f}(\overline{t})=t.
    \]
    From Lemma \ref{lemmamapinducedonwords} and the fact that $\overline{e}$ is well formed, we deduce that $\overline{x}, \overline{y}, \overline{z}, \overline{t},(\overline{x}\ast_k \overline{y}),(\overline{z} \ast_k \overline{t})$ are well formed and from Corollary \ref{corollarycompatiblesubwords}, we deduce that
         \[
     s^k(\overline{x})=t^k(\overline{y}),
     \]
     \[
     s^k(\overline{z})=t^k(\overline{t})
     \]
     and
     \[
     s^l((\overline{x}\ast_k\overline{y}))=t^l((\overline{z}\ast_k\overline{t})).
     \]
     Since $l<k$, we deduce from this last equality that
     \[
     s^l(\overline{x})=s^l(\overline{y})=t^l(\overline{z})=t^l(\overline{t}).
     \]
     Thus, the word
     \[
     \overline{e'}:=((\overline{x}\ast_l\overline{z})\ast_k(\overline{y} \ast_l \overline{t}))
     \]
     is well formed. It satisfies all the desired properties.\qedhere
    \end{description}
  \end{proof}
  \begin{remark}
    In the proof of the previous theorem, we only used the hypothesis that $f$ is right orthogonal to $\kappa_k^n$ for any $k$ such that $0 \leq k < n-1$.
    \end{remark}

\begin{corollary}\label{corollaryisomorphismgraphs}
  Let $f : C \to D$ be an $\omega$-functor, $n>0$, $\Sigma^D \subseteq D_{n}$ and $\Sigma^C = f^{-1}(\Sigma^D)$. If $\tau_{\leq n}(f)$ is a discrete Conduché $n$-functor, then for every $a \in C_{n}$
  \[
  \wt{f}_a : \G[\Sigma^C]_a\to \G[\Sigma^D]_{f(a)}
  \]
  is an isomorphism of graphs.
\end{corollary}
\begin{proof}
 Proposition \ref{propositionequivalentconditionsconduche} exactly says that the map
    \[
  \wt{f}_a : \G[\Sigma^C]_a\to \G[\Sigma^D]_{f(a)}
  \]
  is an isomorphism on objects and we know from Lemma \ref{lemmafaithful} that it is a faithful morphism of graphs (same definition as for functors). All that is left to show is that it is also full.

  In other words, we have to show that for any $u,u' \in \T[\Sigma^C]_a$ the map
  \[
  \G[\Sigma^C](u,u') \to \G[\Sigma^D](\wt{f}(u),\wt{f}(u'))
  \]
  is surjective.

  Let $\mu : \wt{f}(u) \to \wt{f}(u')$ be an element of the codomain. From Lemma \ref{technicallemma} we know that there exists
  \[\lambda : u \to v\]
  in $\G[\Sigma^C]$ such that
  \[\wt{f}(\lambda)=\mu.\]
  In particular, we have
  \[
  \wt{f}(v)=\wt{f}(u').
  \]
  Since we have an elementary move from $u$ to $v$ and by hypothesis ${u \in \T[\Sigma^C]_a}$, we also have $v \in \T[\Sigma^C]_a$ (see \ref{paragr:fiberrho}). Using the injectivity of the map
  \[
  \wt{f}_a : \T[\Sigma^C]_a \to \T[\Sigma^D]_{f(a)}
  \]
  we conclude that $v=u'$.
\end{proof}
  
\begin{proposition}
  Let $f : C \to D$ be an $\omega$-functor, $n \in \mathbb{N}$, $\Sigma^D \subseteq D_{n}$ and $\Sigma^C = f^{-1}(\Sigma^D)$. If $\tau_{\leq n}(f)$ is a discrete Conduché $n$-functor, then:
  \begin{enumerate}
  \item if $\Sigma^D$ is an $n$-basis then so is $\Sigma^C$,
    \item if $f_{n} : C_{n}\to D_{n}$ is surjective and $\Sigma^C$ is an $n$-basis then so is $\Sigma^D$.
    \end{enumerate}
\end{proposition}
\begin{proof}The case $n=0$ is trivial. We know suppose that $n>0$.
  From Corollary \ref{corollaryisomorphismgraphs} we have that for every $a \in C_n$, the map
    \[
    \wt{f}_a : \G[\Sigma^C]_a \to \G[\Sigma^D]_{f(a)}
    \]
    is an isomorphism of graphs. In particular, $\G[\Sigma^C]_a$ is $0$-connected if and only if $\G[\Sigma^D]_{f(a)}$ is $0$-connected. We conclude with Proposition \ref{propositionbasis0connected}.
    \end{proof}
 \begin{theorem}\label{maintheorem}
   Let $f : C \to D$ be a discrete Conduché $\omega$-functor.
   \begin{enumerate}
   \item If $D$ is a free $\omega$-category with basis $(\Sigma^D_n)_{n \in \mathbb{N}}$, then $C$ is a free $\omega$-category with basis $(f^{-1}(\Sigma^D_n))_{n \in \mathbb{N}}$ .
   \item If for every $n \in \mathbb{N}$, $f_n : C_n \to D_n$ is surjective and if $C$ is a free $\omega$-category with basis $(\Sigma^C_n)_{n \in \mathbb{N}}$, then $D$ is a free $\omega$-category with basis $(f(\Sigma^C_n))_{n \in \mathbb{N}}$.
   \end{enumerate}
   \end{theorem}

\begin{proof}
  The first property follow directly from the previous proposition. For the second property, it follows from Lemma \ref{lemmaconducheindecomposablecells} and Proposition \ref{propositionbasisindecomposablecells} that
  \[
 \Sigma^C_n=f^{-1}(f(\Sigma^C_n))
  \]
  and then we can use the previous proposition.
\end{proof}
\appendix
\section{Complements: rigid functors and discrete Conduch\'e $\omega$-functors}\label{appendixrigid}
\begin{paragr}
  We know from Proposition \ref{prop:conducheimpliesrigid} that if
  \[
  f : C \to D
  \]
  is a discrete Conduché $\omega$-functor and if $C$ and $D$ are free $\omega$-categories then $f$ is rigid. However, the converse does not hold. This phenomenon was already noticed for $2$-categories in \cite[section 5]{street1996categorical}. We shall now give a simple counter-example.
\end{paragr}
\begin{counterexample} Let $e$ be the terminal $1$-category and let $\star$ be its unique object. Let $E=(e,\Sigma,\sigma,\tau)$ be the cellular extension of $e$ such that $\Sigma$ has two elements $a,b : \star \to \star$ and let $C:=E^{\ast}$. By the Eckmann-Hilton argument, we have that
  \[
  a \ast_0 b = a \ast_1 b = b \ast_0 a,
  \]
  and $C$ is (isomorphic to) the free \emph{commutative} monoid generated by $a$ and $b$, seen as a 2-category.

  Let $E'=(e,\Sigma',\sigma,\tau)$ be the cellular extension of $e$ such that $\Sigma'$ has one element $ c : \star \to \star$ and let $C':=E'^{\ast}$. Then $C'$ is (isomorphic to) the free commutative monoid generated by $c$, seen as a 2-category.
  Let $f : C \to C'$ be the unique rigid functor such that $f(a)=f(b)=c$. Now set $x:=a \ast_0 b$ and consider the decomposition
  \[
  f(x)=c\ast_0 c.
  \]
  The fact that $a\ast_0 b = b \ast_0 a$ but that $a \neq b$ shows that the uniqueness of the lifting of the previous decomposition of $f(x)$ fails.
\end{counterexample}
\begin{remark}
 While, in the previous counter-example, the uniqueness part of the definition of discrete Conduché $2$-functor fails, the existence still holds. I believe that there should be examples where the existence part fails as well.
\end{remark}
\begin{paragr}
  The category $\Pol$ admits a terminal object $\top$. Hence, a rigid $\omega$-functor $f : C \to D$ between two free $\omega$-categories always fits in a commutative triangle
  \[
  \begin{tikzcd}
    C \ar[rr,"f"] \ar[dr] &&D \ar[dl]\\
    &\top&
    \end{tikzcd}
  \]
  where the anonymous arrows are the canonical rigid functors to the terminal polygraph. Since the class of discrete Conduché $\omega$-functors is a right orthogonal class, it has the following cancellation property: for $f : C \to D$ and $g : D \to E$ two $\omega$-functors, if $g$ and $g \circ f$ are Conduché $\omega$-functors then so is $f$.

  Following the terminology of \cite[section 5]{street1996categorical}, we say that a free $\omega$-category $C$ is \emph{tight} if the canonical rigid functor $C \to \top$ is a Conduché $\omega$-functor. Putting all the pieces together, we obtain the following partial converse of Proposition \ref{prop:conducheimpliesrigid}.  
\end{paragr}
\begin{proposition}\label{prop:partialconverse1}
  Let $C$ and $D$ be two free $\omega$-categories and $f : C \to D$ a rigid $\omega$-functor. If $C$ and $D$ are tight then $f$ is a discrete Conduché $\omega$-functor.
\end{proposition}
\begin{paragr}
  The terminal object of $\Pol$ is a rather complicated object (see \cite[section 4]{street1996categorical} for an explicit description of the 2-cells of that polygraph) and the previous criterion seems hard to use in practise.

  However, it can be checked that every free $1$-category is tight and the previous proposition implies that a $1$-functor $f : C \to D$ between free 1-categories is rigid if and only if it is a Conduché functor. This fact can also be directly proved ``by hand''. We leave the details to the reader.
\end{paragr}

\bibliographystyle{alpha}
\bibliography{fibcond}
\end{document}